\documentclass{amsart}
\usepackage{dsfont}
\usepackage{amsmath}
\usepackage{amsfonts}
\usepackage{amssymb}
\usepackage{amsthm}
\usepackage{tikz, graphicx}
\usepackage{hyperref}
\usepackage{mathtools}
\usepackage[normalem]{ulem}

\usepackage[
top    = 0.9in,
bottom = 0.9in,
left   = 1.0in,
right  = 1.0in]{geometry}
\usepackage{graphicx}

\theoremstyle{plain}
\newtheorem{theorem}{Theorem}[section]

\theoremstyle{definition}

\newtheorem{defi}[theorem]{Definition}
\newtheorem{example}[theorem]{Example}

\newcommand{\C}{\mathds{C}}

\newcommand{\N}{\mathds{N}}

\newcommand{\R}{\mathds{R}}
\newcommand{\Z}{\mathds{Z}}
\newcommand{\calF}{\mathcal F}
\newcommand{\calM}{\mathcal M}
\newcommand{\ds}{\displaystyle}

\renewcommand{\deg}{\operatorname{deg}}

\renewcommand{\o}{\operatorname{o}}
\newcommand{\Lp}[1]{{\mathbf L}^p(#1)}
\newcommand{\Lq}[1]{{\mathbf L}^q(#1)}

\newcommand{\Linfty}[1]{{\mathbf L}^\infty(#1)}
\newcommand\Ker{\operatorname{ker}}

\DeclareMathOperator{\Chi}{Chi}
\DeclareMathOperator{\Par}{par}

\usepackage{xcolor}

\DeclarePairedDelimiter\abs{\lvert}{\rvert}
\DeclarePairedDelimiter\norm{\lVert}{\rVert}

\makeatletter
\let\oldabs\abs
\def\abs{\@ifstar{\oldabs}{\oldabs*}}

\let\oldnorm\norm
\def\norm{\@ifstar{\oldnorm}{\oldnorm*}}
\makeatother

\title[Norms of infinite graphs]{Norms, kernels and eigenvalues of some infinite graphs}

\author[Agrawal et al.]{Aahan Agrawal, Astrid Berge, Seth Colbert-Pollack, Rub\'en
A. Mart\'inez-Avenda\~no and Elyssa Sliheet}

\address{Department of Computer Science\\ University of Illinois at Urbana--Champaign\\ Urbana, Illinois, USA}
\email{aahanagrawal123@gmail.com} 

\address{University of Washington}
\email{astridi1@uw.edu}

\address{Department of Mathematics\\
Kenyon College\\
Gambier, Ohio, USA}
\email{colbertpollacks@kenyon.edu}

\address{Departamento Acad\'emico de Matem\'aticas\\
Instituto Tecnol\'ogico Aut\'onomo de M\'exico \\
Mexico City \\
Mexico}
\email{rubeno71@gmail.com}

\address{Department of Mathematics and Computer Science \\
  Southwestern University \\Georgetown, Texas, USA}
\email{elyssasliheet@gmail.com}

\subjclass[2010]{05C63, 05C50, 47A05, 47A10, 47A75}

\thanks{The results of this research were obtained during the Research
  Experience for Undergraduates Program at California State University
  Channel Islands, in the summer of 2017.  This research was funded by
  NSF grant DMS-1359165 and California State University Channel
  Islands. Rub\'en A. Mart\'inez-Avenda\~no's research is partially
  supported by the Asociaci\'on Mexicana de Cultura A.C}

\begin{document}

\maketitle

\section{Introduction}

The study of infinite graphs, and in particular of their spectral
properties, is relatively new. Perhaps one of the first investigations on the
subject was the 1982 paper by Mohar~\cite{Mohar1}, which examined the
basic properties of the spectrum of the  adjacency matrix of an
infinite graph. His paper gave norm estimates and described other functional
analytic properties. In \cite{MoWo}, Mohar and Woess give a survey
of the results known up to 1989. It seems that many results about the
spectrum of infinite graphs were proven in the context of harmonic
analysis on graphs and on discrete groups, and can be obtained as
corollaries of those results, especially in the case of regular graphs
and regular trees.

Nevertheless, the spectra of infinite graphs has only
recently been studied in greater depth, perhaps motivated by the many
results that have been obtained for the spectrum of the Laplacian of
an infinite graph. For example in \cite{Golinskii1} Golinskii computes the spectra in
$L^2$ of several graphs which are formed by attaching an ``infinite
ray'' to a finite graph. Further results on the spectrum of infinite
graphs can be found in \cite{BiMoSh, DvMo, Golinskii2, LeNy, Nizhnik},
to cite just a few.

In the present paper, we deal with the adjacency matrix of an
infinite, locally finite graph. We consider this matrix as an
operator, which we call the {\em shift operator} (usually called the {\em adjacency operator} in the literature), defined on the set of $p$-summable functions, themselves defined on the vertices of the graph. 

We first obtain some results about the norm of the shift operator: we
show that the shift on the $L^p$ space of a graph $G$ is bounded if and only
if $G$ has bounded degree. We show the norm is bounded by the maximum
degree (Theorem \ref{th:bound}); this result is probably well-known, but we include a proof since it
seems that only the proof for one direction is available in the literature. More
importantly, we show that for a class of graphs, the norm is exactly
the maximum degree (Theorem \ref{th:sharp_bound}). This improves an (unmentioned) corollary of
Theorem 2.1 in  \cite{BiMoSh}, where it is shown that, under
stronger hypotheses than ours, the spectral radius in $L^2$ is the
maximum degree (from this, it follows that the norm equals the maximum
degree). We also obtain a norm estimate for trees in Theorem
\ref{th:bound_tree}; a special case of this theorem appears to have been noticed in the
literature but the proof for the case $p\neq 2$ does not seem to be
available elsewhere. We show  (Theorem \ref{th:sharp_bound_tree}) that this estimate is sharp if the tree
is ``almost regular'', which does not seem to have been observed elsewhere.

We then move to study the triviality of the kernel of the shift
operator on the $L^p$ space of a tree. This question looks to have been
overlooked in the literature (a related result can be found in
\cite{CoMa}). We show (Theorem \ref{th:p_implies_trivial}) that the
kernel of the shift is trivial if some conditions on $p$, on the
``essential'' maximum degree, and on the ``essential minimum'' degree
hold; and we show that these conditions cannot be improved (Example
\ref{ex:p_implies_trivial}).  We also show, in Theorem
\ref{th:p_implies_nontrivial}, that if some conditions hold, then the
kernel of the shift is nontrivial, and we give an example (Example
\ref{ex:p_implies_nontrivial}) which shows that, in some cases, these
conditions cannot be improved. We leave open the question of whether
these conditions can be improved in the rest of the cases.

Lastly, in Section \ref{se:spectrum_tails} we describe an elementary
method to find the eigenvalues of some graphs obtained by attaching a
ray to a finite graph. Results similar to ours have been obtained in
\cite{Golinskii1}, but we believe our method is more elementary. We
give several examples of this method: the kite with an infinite tail,
the fly-swatter with an infinite tail, and the comb with an infinite
tail. We also show that finding these eigenvalues actually gives the
full spectrum of the shift, and finish by finding the spectrum of the infinite comb.

We would like to thank the referee for valuable comments, which improved the presentation of the paper.

\section{Preliminaries}

Let us set the notation that we will use throughout this
article. Recall that a {\em graph} $G$ is a pair $(V,E)$, where $V$ is a
nonempty set, called the set of {\em vertices} of the graph, and $E$ is a
collection of subsets of $V$ of cardinality two. The set $E$ is called the
set of {\em edges} of the graph. If $\{ u, v\} \in E$ we will say thay $u$
and $v$ are {\em adjacent} and denote this relation by $u \sim v$. A
{\em path of length $n$} between two distinct vertices $u$ and $v$ is a finite sequence $\{ u= u_0, u_1,
u_2, \dots u_{n-1}, u_n=v\}$ of distinct vertices such that $u_{j-1} \sim u_{j}$
for $j=1, 2, \dots, n$. We say that the graph is {\em connected} if there
exists a path between any pair of distinct vertices. All of the
graphs we will consider here are assumed to be connected. Clearly
there exists a path of minimum length between any two vertices $u$ and
$v$ and we denote the length of such path by $d(u,v)$, which we define to be zero
if $u=v$. It is clear that $d$ defines a metric on the set of vertices of
the graph.

Here, we will deal mostly with infinite graphs. A graph is {\em
  infinite} if the set $V$ is countably infinite. All of the graphs
here will be {\em locally finite}, which means that the set $\{ v \in V \, :  \:  u
\sim v \}$ is finite for every $u \in V$. The {\em degree} of $u$ is
the cardinality of said set, which we denote by $\deg(u)$. We will say
that a locally finite graph has {\em bounded degree} if the set $\{
\deg(v) \, : \, v \in V \}$ is bounded (this is also called {\em uniformly locally finite} in the literature).

All of our graphs will have a distinguished vertex, which we will
usually denote by the letter $\o$. We define $|v|:=d(\o,v)$. We denote
by $\gamma(n)$ the cardinality of the set $\{ v \in V \, : \, |v|= n\}$.

We say that a graph $T$ is a {\em tree} if for every pair of distinct vertices
there is only one path between them. The distinguished vertex of a
tree is called the {\em root} of the tree and we refer to $T$ as a
{\em rooted tree}. For every vertex $v \neq \o$ in a tree, there
exists a unique vertex $w$ such that $d(v,w)=1$ and $w$ is in the path
from $\o$ to $v$. We call such a vertex the parent of $v$ and denote
it by $\Par(v)$. We define $\Par^n(v)$ inductively: obviously
$\Par^1(v):=\Par(v)$ for every vertex $|v|\geq 1$ and for $n \in \N$, with
$n \geq 2$ we define $\Par^n(v):=\Par(\Par^{n-1}(v))$ for every vertex $|v|\geq n$.
For every $v \in V$, the set $\{ u \in V \, : u \sim v, u \neq \Par(v) \}$ is
called the set of {\em children} of $v$ and is denoted by
$\Chi(v)$. Also, for every $v \in V$ and $n \in \N$ with $n \geq 2$, we set
\[
\Chi^n(v):=\{ u \in V \: \, \Par^n(u)=v \}.
\] 
If $\Chi(v)$ is empty for some $v \in V$, we say that $v$ is a {\em
  leaf} of $T$. If $T$ has no leaves, we say the tree is leafless.

We denote the vector space of all functions $f : V \to \C$ as
$\calF$. As is customary, we denote by $\Lp{G}$ the set
\[
\{ f \in \calF \, : \, \| f \|_p:= \left(\sum_{v \in V} |f(v)|^p \right)^{1/p} < \infty \}
\]
if $1 \leq p < \infty$, and by $\Linfty{G}$ the set
\[
\{ f \in \calF \, : \, \| f \|_\infty:= \sup_{v \in V} |f(v)| < \infty \}.
\]
Clearly $\Lp{G}$ is a Banach space, since it is isomorphic to
$\ell^p(V)$.

Our main object of study is the {\em shift operator} $S$. This is
defined on $\calF$ as
\[
(Sf)(u)=\sum_{v \sim u} f(v). 
\]
If the graph is finite, the matrix of $S$ with respect to the
canonical basis of the vector space  $\calF$ is the
well-known and much-studied {\em adjacency} matrix of the graph.

We denote by $\chi_A$ the characteristic function of the set $A
\subset V$. Clearly $\chi_A$ is in $\Linfty{G}$ but, for $1\leq p <
\infty$, the function $\chi_A$ is in $\Lp{G}$ if and only if $A$ is
finite. If $A$ is a singleton $\{v\}$ we write $\chi_v:=\chi_A$.

The following version of Jensen's inequality will be used. If
$\{a_1, a_2, \dots, a_n \}$ are nonnegative numbers, then
\[
\left(\frac{a_1+a_2+\dots + a_n}{n}\right)^p \leq \frac{a_1^p+a_2^p+\dots + a_n^p}{n}
\]
for $1\leq p < \infty$. We will also use Young's inequality: if $a, b
\geq 0$ then $ a b \leq \frac{a^p}{p} + \frac{b^q}{q}$
for $1<  p < \infty$ and $q=\frac{p}{p-1}$. Lastly, we will use the
following inequality, which follows from the convexity of the function
$t \mapsto t^p$ (for $1 \leq p < \infty$) and a trivial induction argument: if
$\{a_1, a_2, \dots, a_n \}$ are nonnegative numbers, then
\[
\left(a_1+a_2+\dots + a_n\right)^p \geq a_1^p+a_2^p+\dots + a_n^p
\]
for $1\leq p < \infty$.

From now on every graph is assumed to be locally finite and connected.

\section{Boundedness}

In this section we study the norm of the shift. This has been
studied for the case $p=2$ in several places (see for example, Theorem
3.2 in \cite{Mohar1} or Theorem 3.1 in \cite{MoWo}). The ``if'' part
of the following result can be found in \cite{vB}, but we have not been
able to find a proof of the ``only if'' part for arbitrary $p \geq 1$,
although it is probably well-known. For completeness, we have decided
to include here the proof of both directions (our proof differs from
the proof of Theorem 3.1 in \cite{vB}).

\begin{theorem}\label{th:bound}
Let $G=(V,E)$ be a graph and let $1\leq p \leq \infty$. Then $S$ is
bounded in $\Lp{G}$ if and only if $G$ has bounded degree. In this case,
\[
\norm{S} \leq \max\{\deg(v) \, : \, v \in V\}.
\]
\end{theorem}
\begin{proof} 
First, assume that $G$ has bounded degree and let $M:= \max\{\deg(v)
\, : \, v \in V\}$. 

\begin{enumerate}
\item Suppose $1 \leq p < \infty$. Then, for every $f \in \Lp{G}$
\begin{align*}
\norm{Sf}_{p}^p = \sum_{v \in V}\abs{\sum_{u \sim v} f(u)}^p 
&\leq \sum_{v \in V}  \left(\sum_{u \sim v}\abs{f(u)}\right)^p  \\
&\leq \sum_{v \in V} \deg(v)^{p-1} \sum_{u \sim v}\abs{f(u)}^p \quad \text{
  (by Jensen's inequality) }\\
&\leq   M^{p-1} \sum_{v \in V} \sum_{u \sim v} \abs{f(u)}^p \\
&\leq M^p \norm{f}_{p}^p. 
\end{align*}
Hence $\norm{Sf}_p^p \leq M^p \norm{f}_p^p$ for all $f \in \Lp{G}$ which
implies that $S$ is a bounded operator from $\Lp{G}$ into $\Lp{G}$ and $\norm{S} \leq M$.

\item Suppose $p = \infty$. Let $f \in \Linfty{G}$. For every $v \in V$
  we have
\[
\abs{Sf(v)} = \sum_{u \sim v}\abs{f(u)} \leq M \norm{f}_\infty.
\]
Taking the supremum over all $v \in V$ we obtain $\norm{Sf}_\infty
\leq M \norm{f}_\infty$ and hence  $S$ is a bounded  operator from
$\Lp{G}$ into $\Lp{G}$ and $\norm{S} \leq M$.

\end{enumerate}

Now assume $G$ has unbounded degree.
\begin{enumerate}
\item Suppose $1 \leq p < \infty$. Since $G$ has unbounded degree,
  there exists a sequence of vertices $\{ v_n  \}$ such that
  $\deg(v_n) \geq 2^{np}$.

Let $f \in \calF$ be a function such that for all $v \in V$, 
\[
f(v)= \begin{cases}
        2^{-n}, & \text{ if }v=v_n,\\
        0, & \text{ otherwise. }
        \end{cases}
\]
Then 
\[
\norm{f}_p^p = \sum_{v\in V} |f(v)|^p = \sum_{n=1}^\infty
|2^{-n}|^p < \infty
\]
so $f \in \Lp{G}$. But
\[
\norm{Sf}^p_p  = \sum_{v \in V} \left(\sum_{u \sim v} f(u) \right)^p \geq
\sum_{v \in V} \sum_{u \sim v} (f(u))^p = \sum_{n=1}^\infty
\deg(v_n) (f(v_n))^p \geq \sum_{n=1}^n 2^{np}
\left(\frac{1}{2^n}\right)^p  = \infty.
\]
Hence $S$ is unbounded.

\item Now suppose $p= \infty$. Let $f \in \calF$ such that $f(v)=1$ for
  all $v \in V$. Then $\|f\|_\infty =1$, so $f \in \Linfty{G}$. But 
\[
\| S f \| = \sup_{v \in V}\{ | \sum_{u \sim v} f(u)| \} =
\sup_{v \in V} |\deg(v)|=\infty,
\]
since the degree of $G$ is unbounded. Hence $S$ is unbounded. \qedhere
\end{enumerate}
\end{proof}

In addition to the last theorem, observe (this had already been
noticed in \cite[Theorem 3.1]{vB}) that in the case $p=1$ and
$p=\infty$, we have $\norm{S}=\max\{ \deg(v) \, : \, v \in
V\}$. Indeed, let $v$ be a vertex with $k:=\deg(v)=\max\{ \deg(u) \, :
\, u \in V \}$. If $p=1$,  then $\norm{\chi_{v}}_1=1$ and $\norm{S
  \chi_{v}}_1= k$, which shows that $\norm{S}=k$. For $p=\infty$, let
$A=\{ u \in V \, : \, u \sim v \}$. Then $\norm{\chi_A}_\infty=1$ and
$\norm{S \chi_A}_\infty=k$, which shows that $\norm{S}=k$.

The natural question is whether the bound for the norm of the shift
found above is optimal also for the cases $1< p < \infty$. We now
give several definitions which allow us to introduce a class of
graphs for which the bound will be optimal.

\begin{defi}\label{def_euc}
Let $G=(V,E)$ be an infinite graph. 
\begin{itemize} 
\item We say $G$ is {\em almost $k$-regular} if
$\deg(v) \leq k$ for all $v \in V$ and there exists a finite set of
vertices $V_0$ such that $\deg(v)=k$ for all $v \in V \setminus V_0$.

\item For the following definition, recall that, for $n \in \N$,
  $\gamma(n)$ is defined to be the number of vertices at distance $n$
  from the distinguished vertex. We say $G$ is {\em
    $k$-almost-Euclidean} if the graph is almost $k$-regular,
    and if
\begin{equation}\label{kEuclidean1}
\lim_{n\to \infty} \frac{\gamma(n) + \gamma(n+1)}{\gamma(0)+\gamma(1)+\gamma(2)+\dots + \gamma(n)}=0.
\end{equation}
\end{itemize}
\end{defi}
  Let us make a few comments on this definition. First of all, observe
  that, since the graph is infinite, $\ds \lim_{n\to \infty} \gamma(0)
  + \gamma(1)+\gamma(2)+\dots + \gamma(n)=\infty$. Hence, for any $m
  \in \N_0$ we have
\begin{equation}\label{eq:limit}
\lim_{n\to \infty} \frac{\gamma(0) + \gamma(1)+\gamma(2)+\dots
  +\gamma(m)}{\gamma(0) + \gamma(1)+\gamma(2)+\dots + \gamma(n)}=0.
\end{equation}
Also, notice that if the expression \eqref{kEuclidean1} holds, we get
\begin{equation}\label{kEuclidean_n}
\lim_{n\to \infty}
\frac{\gamma(n)}{\gamma(0)+\gamma(1)+\gamma(2)+\dots + \gamma(n)}=0.
\end{equation}
and hence, for $0\leq m < n$ we obtain
\begin{equation}\label{kEuclidean2}
\lim_{n\to \infty} \frac{\gamma(m)+\gamma(m+1)+\dots +
  \gamma(n-1)}{\gamma(0)+\gamma(1)+\gamma(2)+\dots + \gamma(n)}=1.
\end{equation}

Observe that if $\gamma$ satisfies $\gamma(n+1) \geq \gamma(n)$ for
all $n$ greater than some $n_0 \in \N$ and 
\begin{equation}\label{kEuclidean3}
\lim_{n\to \infty}
\frac{\gamma(n+1)}{\gamma(0)+\gamma(1)+\gamma(2)+\dots +
  \gamma(n)}=0,
\end{equation}
then expression \eqref{kEuclidean1} is satisfied.
(Hence the name of the second part of Definition \ref{def_euc}, which is based on the definition in \cite{BiMoSh}.)

Note added: Professor L.~Golinskii has kinldy pointed to us that if there exists a
constant $c>0$ such that $\gamma$ satisfies $\gamma(n+1) \geq c
\gamma(n)$ for all $n$ greater than some $n_0 \in \N$ and expression
\eqref{kEuclidean3} is satisfied, then expression \eqref{kEuclidean_n} is
also satisfied. Hence, the limit \eqref{kEuclidean1} also holds.

We can now prove that for a certain class of graphs, the norm of the
shift equals the maximum degree of the graph.

\begin{theorem}\label{th:sharp_bound}
Let $G=(V,E)$ be a graph, let $1 \leq p < \infty$ and let $S:\Lp{G} \to
\Lp{G}$ be the shift operator. If $G$ is $k$-almost-Euclidean, then $\norm{S}=k$.
\end{theorem}
\begin{proof}
  Let $\o$ be the distinguished vertex of $G$.
Choose $m \in \N$ large enough such that if $|v|\geq m-1$, then
$\deg(v)=k$. For every $n \in \N$, $n > m$,  define $f_n:V \to \C$ as $f_n(v)=1$ if $|v|\leq n$ and $f_n(v)=0$
otherwise.

First, observe that if $m\leq |v|<n$ then $(Sf_n)(v)=k$, and if $|v|> n+1$ then $(Sf_n)(v)=0$. 

It then follows that, for $n > m$, 
\[
\norm{S f_n}_p^p = \sum_{v \in V} |(S f_n)(v) |^p = \sum_{|v|< m}
\abs{ (Sf_n)(v)}^p + \sum_{m\leq |v|<n} \abs{ (Sf_n)(v)}^p + \sum_{|v| =
  n} \abs{ (Sf_n)(v)}^p + \sum_{|v| = n+1} \abs{(S f_n)(v)}^p .
\]
We have
\[
\sum_{m\leq |v|<n} \abs{ (Sf_n)(v)}^p = \sum_{m\leq|v|<n} \abs{k}^p =
k^p (\gamma(m)+\gamma(m+1)+\dots + \gamma(n-1)).
\]
We also obtain the inequalities
\[
\sum_{|v| = n} \abs{ (Sf_n)(v)}^p + \sum_{|v| = n+1} \abs{(S
  f_n)(v)}^p \leq \sum_{|v| = n} \abs{k}^p + \sum_{|v| = n+1}
\abs{k}^p  = k^p (\gamma(n) + \gamma(n+1)).
\]
and 
\[
\sum_{|v| < m} \abs{ (Sf_n)(v)}^p   \leq \sum_{|v| < m}
\abs{k}^p = k^p (\gamma(0)+\gamma(1)+\gamma(2)+ \dots + \gamma(m-1)).
\]
Since $\norm{f_n}_p^p= \gamma(0)+\gamma(1)+\gamma(2)+ \dots +
  \gamma(n)$, by expression \eqref{kEuclidean2} we have
\[
\lim_{n\to \infty} \frac{\sum\limits_{m\leq |v|<n}
  \abs{(Sf_n)(v)}^p}{\norm{f_n}_p^p}= \lim_{n\to \infty} \frac{k^p
  (\gamma(m)+\gamma(m+1)+\dots +
  \gamma(n-1))}{\gamma(0)+\gamma(1)+\gamma(2)+ \dots + \gamma(n))} = k^p.
\]
Also, since the graph is $k$-almost-Euclidean by expression
  \eqref{eq:limit} we have
\[
0 \leq \lim_{n\to \infty} \frac{\sum\limits_{|v| \leq m} \abs{
    (Sf_n)(v)}^p}{\norm{f_n}_p^p} \leq
\lim_{n\to \infty} \frac{k^p
  (\gamma(0)+\gamma(1)+\gamma(2)+ \dots +
  \gamma(m))}{\gamma(0)+\gamma(1)+\gamma(2)+ \dots + \gamma(n)} = 0,
\]
Using the same argument, we have by the limit \eqref{kEuclidean1} that
\[
\lim_{n\to \infty} \frac{\sum\limits_{|v|=n} \abs{(Sf_n)(v)}^p +
  \sum\limits_{|v|=n+1} \abs{(Sf_n)(v)}^p}{\norm{f_n}_p^p}=0.
\]

Hence,
\[
\lim_{n\to \infty} \frac{\norm{S f_n}_p^p}{\norm{f_n}_p^p} = k^p,
\]
and therefore $\norm{S} \geq k$. By the previous theorem, we obtain
the equality.
\end{proof}

\begin{example}
    For every $d \in \N$, the lattice $\Z^d$ (where two points are adjacent if and only if their Euclidean distance in $\R^d$ is $1$) is $2d$-almost-Euclidean.
  \end{example}
  \begin{proof}
    The lattice $\Z^d$ is obviously $2d$-regular. Set the origin $\o$ as the distinguished vertex. Clearly $\gamma(0)=1$. The set of points at distance $1$ from $\o$ are the vertices of a $d$-dimensional octahedron (i.e., the $1$-skeleton of a cross-polytope or orthoplex; see, e.g., \cite[\S 7.2]{Cox}): hence, there are $2d$ of them. The set of points at distance $2$ from $\o$ are the vertices of a $d$-dimensional octahedron and the midpoints of its edges. Hence, since there are $2d(d-1)$ edges, we have $\gamma(2)=2d(d-1)+2d$. Analogously, the set of points at distance $3$ from $\o$ are the vertices of a $d$-dimensional octahedron and the points obtained by trisecting its edges: hence $\gamma(3)=2(2d(d-1))+2d$. In general,  the set of points at distance $n$ from $\o$ are the vertices of a $d$-dimensional octahedron and the points obtained by dividing its edges in $n$ equal parts. Hence we have
    \[
      \gamma(n)=(n-1)(2d(d-1))+2d.
    \]
    It is then straightforward to check that
    \[
      \lim_{n\to \infty} \frac{\gamma(n) + \gamma(n+1)}{\gamma(0)+\gamma(1)+\gamma(2)+\dots + \gamma(n)}=0.
    \]
    and hence the lattice $\Z^d$ is $2d$-almost-Euclidean.
    \end{proof}

\begin{example}
The triangular tesselation of the plane is $6$-almost-Euclidean and the hexagonal tesselation of
the plane is $3$-almost-Euclidean.
\end{example}
\begin{proof}
  The triangular tesselation of the plane is a $6$-regular graph and we have $\gamma(0)=1$ and $\gamma(n)=6n$. The hexagonal tesselation of the plane is a $3$-regular graph and we have $\gamma(0)=1$ and $\gamma(n)=3n$. The result follows.
\end{proof}

We should also point out that the graphs obtained by removing finitely many edges from the previous graphs (while keeping
them connected) are also $k$-almost-Euclidean.

Also, any almost $k$-regular graph for which $\gamma$ is bounded is $k$-almost-Euclidean. For example, the semi-infinite ladder graph is $3$-almost-Euclidean (thanks to Professor L.~Golinskii for pointing out this last example).

There is a class of graphs for which there is a smaller bound for the
norm: the trees. The following theorem is well-known for $p=2$, but the case $p\neq 2$
seems to have gone unnoticed. First recall that if $1 < p < \infty$
and $q:=\frac{p}{p-1}$, the duality between $\Lp{G}$ and
$\Lq{G}$ is realized by the pairing
\[
\left< f, g\right> =\sum_{v\in V} f(v) g(v),
\]
where $f \in \Lp{G}$ and $g \in \Lq{G}$.

Observe that the following result, if $k=2$, is included in Theorem~\ref{th:bound}. In fact, if $k=2$, equality follows from Theorem~\ref{th:sharp_bound} since in this case the tree is $2$-almost-Euclidean.

\begin{theorem}\label{th:bound_tree}
Let $T=(V,E)$ be a tree such that there exists $k\geq 2$ with $\deg(v) \leq k$ for every $v \in V$. Let $1<p< \infty$ and let
$q=\frac{p}{p-1}$. Then $\| S \| \leq  (k-1)^{1/p} + (k-1)^{1/q}$.
\end{theorem}
\begin{proof}
Let $\o$ be the distinguished vertex of $T$; i.e., the root of $T$.
Let $f \in \Lp{T}$ and $g\in \Lq{T}$ with $\|f \|_p\leq 1$ and $\|g
\|_q \leq 1$. Since $T$ is a tree we can write
\begin{align*}
\left< S f, g \right> 
&= \sum_{v\in V} \left(\sum_{w \sim v} f(w)\right) g(v) \\
&= \sum_{v\in V} \left(\sum_{w \in \Chi(v)} f(w)\right) g(v) +
 \sum_{v\in V\setminus\{\o\}} f(\Par(v)) g(v) \\
&= \sum_{v\in V} \left(\sum_{w \in \Chi(v)} f(w)\right) g(v) + \sum_{v\in
       V} f(v) \left(\sum_{w \in \Chi(v)} g(w) \right).
\end{align*}
Applying the triangle inequality and then Young's inequality to each summand we obtain
\begin{equation}\label{eq:Sfg}
  \begin{split}
|\left< S f, g \right>| 
\leq {} & \sum_{v\in V} \sum_{w \in \Chi(v)} |f(w)| |g(v)| + \sum_{v\in
       V} \sum_{w \in \Chi(v)} |f(v)| |g(w)| \\
\leq {} &  \sum_{v\in V} \sum_{w \in \Chi(v)} \left( \tfrac{1}{p} | f(w)|^p
         (k-1)^{p/(p+q)} + \tfrac{1}{q} |g(v)|^q (k-1)^{-q/(p+q)} \right) \\
& + \sum_{v\in V} \sum_{w \in \Chi(v)} \left( \tfrac{1}{p} | f(v)|^p (k-1)^{-p/(p+q)} + \tfrac{1}{q} |g(w)|^q (k-1)^{q/(p+q)} \right) \\
  = {} &  \frac{(k-1)^{p/(p+q)}}{p} \sum_{v\in V} \sum_{w \in \Chi(v)} | f(w)|^p +  \frac{(k-1)^{-q/(p+q)}}{q} \sum_{v\in V} \sum_{w \in \Chi(v)}  |g(v)|^q \\
& + \frac{(k-1)^{-p/(p+q)}}{p} \sum_{v\in V} \sum_{w \in \Chi(v)} |f(v)|^p  +  \frac{(k-1)^{q/(p+q)}}{q} \sum_{v\in V} \sum_{w \in \Chi(v)}  |g(w)|^q.
\end{split}
\end{equation}

Since the root has at most $k$ children and every other vertex has at most $k-1$ children, we have
\[
  \sum_{v\in V} \sum_{w \in \Chi(v)}  |g(v)|^q \leq (k-1) \sum_{v \in V\setminus\{\o\}} |g(v)|^q + k |g(\o)|^q
  = (k-1) \sum_{v \in V} |g(v)|^q + |g(\o)|^q = (k-1) \| g\|_q^q + |g(\o)|^q 
\]
and
\[
  \sum_{v\in V} \sum_{w \in \Chi(v)}  |f(v)|^p \leq (k-1) \sum_{v \in V\setminus\{\o\}} |f(v)|^p + k |f(\o)|^p = (k-1) \sum_{v \in V} |f(v)|^p + |f(\o)|^p = (k-1) \| f\|_p^p + |f(\o)|^p;
\]
also,
\[
  \sum_{v\in V} \sum_{w \in \Chi(v)} | f(w)|^p = \| f \|_p^p - |f(\o)|^p \quad\text{ and }\quad \sum_{v\in V} \sum_{w \in \Chi(v)} |g(w)|^q = \| g \|_q^q - |f(\o)|^q.
\]

Substituting the previous expressions into inequality \eqref{eq:Sfg} we obtain
\begin{align*}
|\left< S f, g \right> | 
 \leq {} & \frac{(k-1)^{p/(p+q)}}{p} \left( \|f\|_p^p -| f(\o)|^p \right) + \frac{(k-1)^{-q/(p+q)}}{q}  \left( (k-1) \| g \|_q^q + |g(\o)|^q \right) \\
 & + \frac{(k-1)^{-p/(p+q)}}{p} \left( (k-1) \| f\|_p^p + |f(\o)|^p \right)
+  \frac{(k-1)^{q/(p+q)}}{q} \left( \| g\|_q^q - |g(\o)|^q \right) \\
  = {} & \left( \frac{(k-1)^{p/(p+q)}}{p}  + \frac{(k-1)^{1-p/(p+q)}}{p}  \right) \| f \|_p^p + \left( \frac{(k-1)^{q/(p+q)}}{p}  + \frac{(k-1)^{1-q/(p+q)}}{p}  \right) \| g \|_q^q \\
         & - \left( \frac{(k-1)^{p/(p+q)}}{p} - \frac{(k-1)^{-p/(p+q)}}{p} \right) |f(\o)|^p
           - \left( \frac{(k-1)^{q/(p+q)}}{p} - \frac{(k-1)^{-q/(p+q)}}{q} \right) |g(\o)|^p \\
  \leq {} & \left( \frac{(k-1)^{p/(p+q)}}{p}  + \frac{(k-1)^{q/(p+q)}}{p}  \right) \| f \|_p^p + \left( \frac{(k-1)^{q/(p+q)}}{p}  + \frac{(k-1)^{p/(p+q)}}{p}  \right) \| g \|_q^q,
\end{align*}
where we used the fact that $1-\frac{q}{p+q}=\frac{p}{p+q}$ and $1-\frac{p}{p+q}=\frac{q}{p+q}$ and that
\[
  \frac{(k-1)^{p/(p+q)}}{p} - \frac{(k-1)^{-p/(p+q)}}{p} \geq 0 \quad \text{ and } \quad  \frac{(k-1)^{q/(p+q)}}{p} - \frac{(k-1)^{-q/(p+q)}}{q} \geq 0. 
\]

Recalling that $\frac{p}{p+q}=\frac{1}{q}$ and $\frac{q}{p+q}=\frac{1}{p}$ we obtain
\[
|\left< S f, g \right> | 
\leq  \frac{(k-1)^{1/q}}{p}  +  \frac{(k-1)^{1/p}}{p}  +
\frac{(k-1)^{1/p}}{a}  +  \frac{(k-1)^{1/q}}{q}
= (k-1)^{1/q}  + (k-1)^{1/p},
\]
since $\|f\|_p \leq 1$ and $\|g\|_q\leq 1$.

Since $\| S \| = \sup\{ |\left< S f, g \right>| \, : \, \|f \|_p\leq
1, \|g \|_q\leq 1 \}$ (see, for example, Proposition 1.10.11 in
\cite{Megginson}), we obtain the desired result.
\end{proof}

In some cases, the bound for the norm above is attained.

\begin{theorem}\label{th:sharp_bound_tree}
Let $T=(V,E)$ be an almost $k$-regular rooted tree. Let $1<p< \infty$ and let
$q=\frac{p}{p-1}$. Then $\| S \| =   (k-1)^{1/p} + (k-1)^{1/q}$.
\end{theorem}
\begin{proof}
Since $T$ is almost $k$-regular, there exists $N \in \N$ such that
$\deg(v)=k$ if $|v|\geq N$. For each $n > N$ we define the function
\[
f_n(v)= \begin{cases}
(k-1)^{-\frac{|v|}{p}}, & \text{ if } N < |v| \leq n, \\
0, & \text{ otherwise.}
\end{cases}
\]
Clearly $f_n \in \Lp{G}$ for each $n \in \N$, $n>N$.  

Let $C=\gamma(N)$. It then follows that for each $j \geq N$, there are $C (k-1)^{j-N}$ vertices
$v$ with $|v|=j$. With this in mind, it follows that
\[
\| f_n \|_p^p = \sum_{j=N+1}^{n} C (k-1)^{j-N} (k-1)^{-j} = (n-N) C (k-1)^{-N}.
\]
We also have that
\begin{align*}
\| S f_n \|_p^p 
= {} & C \abs{ 1\cdot 0 + (k-1) (k-1)^{-\frac{N+1}{p}} }^p + C (k-1) \abs{
    1 \cdot 0 + (k-1) (k-1)^{-\frac{N+2}{p}}}^p \\
 & + \sum_{j=N+2}^{n-1} C (k-1)^{j-N} \abs{(k-1)^{-\frac{j-1}{p}} + (k-1)
    (k-1)^{-\frac{j+1}{p}} }^p \\
 & + C (k-1)^{n-N} \abs{ (k-1)^{-\frac{n-1}{p}} + (k-1) 0 }^p  + C (k-1)^{n+1-N} \abs{(k-1)^{-\frac{n}{p}} + (k-1) 0 }^p.
\end{align*}
Simplifying yields
\begin{align*}
\| S f_n \|_p^p 
={}& C \bigg( (k-1)^{p-(N+1)} + (k-1)^{1+ p-(N+2)} \\
&  + \sum_{j=N+2}^{n-1} (k-1)^{-N }\abs{(k-1)^{\frac{j}{p}-\frac{j-1}{p}}
    + (k-1) (k-1)^{\frac{j}{p} -\frac{j+1}{p}} }^p   \\
&  +  (k-1)^{-N+1} + (k-1)^{-N+1} \bigg) \\
={} & C (k-1)^{-N}\bigg( 2 (k-1)^{p-1}  
     +  \sum_{j=N+2}^{n-1} \abs{(k-1)^{\frac{1}{p}} +  (k-1)^{1-\frac{1}{p}} }^p  + 2 (k-1) \bigg) \\
={} & C (k-1)^{-N} \bigg( 2 (k-1)^{p-1} 
+  ((n-1)-(N+1))
 \abs{(k-1)^{\frac{1}{p}} +  (k-1)^{1-\frac{1}{p}} }^p  +  2 (k-1) \bigg).
\end{align*}
From this, it follows that
\[
\frac{\| S f_n \|_p^p}{\| f_n \|_p^p}= \frac{
 2 (k-1)^{p-1}  +  (n-N-2) \abs{(k-1)^{\frac{1}{p}} +  (k-1)^{1-\frac{1}{p}} }^p  +  2 (k-1)
}{n-N}.
\]
Therefore
\[
\lim_{n \to \infty} \frac{\| S f_n \|_p^p}{\| f_n \|_p^p} = \abs{(k-1)^{\frac{1}{p}} +  (k-1)^{1-\frac{1}{p}} }^p,
\]
and hence
\[
\| S \| \geq
(k-1)^{\frac{1}{p}} +  (k-1)^{\frac{1}{q}}.
\]
By the previous theorem, it follows that
\[
\| S \| =  (k-1)^{\frac{1}{p}} +  (k-1)^{\frac{1}{q}}
\]
as desired.
\end{proof}


\section{Kernel of trees}

In this section, we try to answer the question of when the kernel of
the shift is trivial or nontrivial in the case of leafless rooted trees. This question seems to have been overlooked in the literature and it yields some interesting results. Some results for a modification of the shift can be found in \cite{CoMa}.

First, we need some definitions.

\begin{defi}
Let $T=(V,E)$ be a rooted tree and let $v \in V$. We define $\beta(v)$ to be
the number of children of $v$.
\end{defi}

Observe that $\beta(v)=\deg(v)-1$ for every vertex $v\neq \o$ and $\beta(\o)=\deg(\o)$. Also,
$\beta(v)=0$ if and only if $v$ is a leaf.

\begin{defi}
Let $T=(V,E)$ be a rooted leafless tree with bounded degree. Consider
the set $\calM$ defined as
\[
\calM:= \{ v \in V \, : \, \beta(v)=\beta(w)  \text{ for infinitely
  many } w \in V \}.
\]
Since the tree has bounded degree, the set $\{\beta(v) \, : \, v \in V\}$ is finite and hence the set $V\setminus \calM$ is finite too. We define the numbers
\[
M:= \max_{v \in \calM}\{\beta(v)\}, \quad \text{ and } \quad m:=
\min_{v \in \calM}\{\beta(v)\}.
\]
Observe that $m=0$ if and only if $T$ has infinitely many leaves.
\end{defi}

For the rest of this section, we will only consider leafless trees. We will say a few words at the end of this section about the type of behavior that may occur if the tree has leaves.

We will need the following key observation. If $T=(V,E)$ is a leafless
tree and if $f \neq 0$ satisfies $Sf=0$, then there must exist $v^* \in V$
such that $f(v^*) \neq 0$. Let $w^* \in \Chi(v^*)$ (the set $\Chi(v^*)$ is nonempty, since $T$ is leafless). Since
$(Sf)(w^*)=0$, we must have 
\[
0=(Sf)(w^*)=f(v^*) + \sum_{u \in \Chi(w^*)} f(u),
\]
and hence there exists $u^* \in \Chi^2(v^*)$ with $f(u^*)\neq 0$. That
is, if $f \in \ker(S)$ and $f$ is not zero at a vertex, then there is
a grandchild of the vertex in which $f$ does not vanish. Applying this
argument inductively, it is clear that for all $n \in \N$ there exists $u_n
\in \Chi^{2n}(v^*)$ with $f(u_n) \neq 0$.

We are now ready to prove the following theorem.

\begin{theorem}\label{th:p_implies_trivial}
Let $T=(V,E)$ be a leafless tree with bounded degree and consider the
shift $S$ on $\Lp{T}$, with $1 \leq p  \leq \log_M(m)+1$ if $M>1$, or $1\leq p < \infty$ if $M=1$. Then
$\ker(S)$ is trivial.
\end{theorem}
\begin{proof}
Assume there exists $f\in\ker(S)$ such that $f\neq0$. Then there
exists $v^* \in V$ such that $f(v^*)\neq 0$.

Since $T$ has bounded degree, there exists $N \in \N$ such that for
all $|v|\geq 2N$, we have $v \in \calM$. We claim that
\begin{equation}\label{eq:treesum_trivial}
C \left(\frac{m}{M^{p-1}}\right)^k \leq \sum_{u \in
    \Chi^{2k}(v^*)} \abs{f(u)}^p
\end{equation}
for all $k \in \N$ with $k \geq N$, where
\[
C:= \left(\frac{m}{M^{p-1}}\right)^{-N} \sum_{u \in
    \Chi^{2N}(v^*)} \abs{f(u)}^p.
\]

To prove the claim, observe first that, since $f(v^*) \neq 0$, the observation before this
theorem implies that there exists $u \in \Chi^{2N}(v^*)$ such that
$f(u)\neq 0$. Hence, $C$ is positive.

We proceed by induction on $k$. Observe that for $k=N$ the inequality
\eqref{eq:treesum_trivial} is trivially satisfied. Now, assume that
for a fixed $k \in \N$, with $k\geq N$ the inequality
\eqref{eq:treesum_trivial} is satisfied.

Let $w^* \in \Chi^{2k}(v^*)$ and let $v_0$ be a child of $w^*$. Then,
since $(Sf)(v_0)=0$ we have
\[
0=f(w^*)+{\sum_{u \in \Chi(v_0)}f(u)}.
\]
It then follows that
\[
\abs{f(w^*)} \leq \sum_{u \in \Chi(v_0)}\abs{f(u)},
\]
and therefore, by Jensen's inequality and since $\beta(v_0) \leq M$, we have
\[
\abs{f(w^*)}^p \leq \beta(v_0)^{p-1} \sum_{u \in \Chi(v_0)}\abs{f(u)}^p
\leq M^{p-1} \sum_{u \in \Chi(v_0)}\abs{f(u)}^p.
\]
The above inequality holds for all $v_0 \in \Chi(w^*)$. We know $w^*$ has at least $m$ children, so 
\[ 
m \abs{f(w^*)}^p \leq \sum_{v_0 \in \Chi(w^*)} \abs{f(w^*)}^p \leq
M^{p-1} \sum_{v_0 \in \Chi(w^*)} \sum_{u\in \Chi(v_0)} \abs{f(u)}^p =
M^{p-1} \sum_{u\in \Chi^2(w^*)} |f(u)|^p,
\]
and therefore
\[ 
\frac{m}{M^{p-1}} \abs{f(w^*)}^p \leq  \sum_{u\in \Chi^2(w^*)} |f(u)|^p.
\]

The above holds for all $w^*\in \Chi^{2k}(v^*)$. By the induction hypothesis we have
\begin{align*}  
C \left(\frac{m}{M^{p-1}} \right)^{k+1} 
&\leq \frac{m}{M^{p-1}} \sum_{w^*\in \Chi^{2k}(v^*)}\abs{f(w^*)}^p \\
&= \sum_{w^*\in \Chi^{2k}(v^*)}  \frac{m}{M^{p-1}}|f(w^*)|^p \\
&\leq \sum_{w^* \in \Chi^{2k}(v^*)}\sum_{u \in \Chi^{2} (w^*)} \abs{f(u)}^p\\
&= \sum_{u \in \Chi^{2(k+1)}(v^*)} \abs{f(u)}^p.
\end{align*}
Thus, our claim is true by induction. 

To finish the proof, observe that
\begin{align*}
   \sum_{w \in V}\abs{f(w)}^p &\geq \sum_{k=N}^\infty \sum_{u\in \Chi^{2k}(v^*)} |f(u)|^p \\
   &\geq C \sum_{k=N}^\infty  \left(\frac{m}{M^{p-1}} \right)^{k}.
\end{align*}
Since $\frac{m}{M^{p-1}}\geq 1$ by hypothesis, the above series diverges,
which is a contradiction, since $f \in \Lp{T}$. Hence $\ker(S)=\{0\}$
for $1\leq p \leq 1 + \log_M(m)$, if $M>1$ and for $1 \leq p < \infty$, if $M=1$.
\end{proof}

The following example shows that the bound for $p$ in the theorem above cannot be improved.

\begin{example}\label{ex:p_implies_trivial}
Let $M\geq m$ be natural numbers and construct a tree $T$ as
follows. Let $\o$ be the root and add $m$ children to it. To each of
these new vertices add $M$ children. To each of the latter vertices
add $m$ children, and to each of those vertices add $M$
children. Continue in this manner.  
Consider the shift $S$ on
$\Lp{T}$, with $\log_M(m)+1<p \leq \infty$ if $M>1$ or $p=\infty$ if
$M=1$. Then $\ker(S)$ is nontrivial.

\begin{figure}[h]
\newcommand\vrad{0.8mm}
\newcommand\ethick{0.1pt}

\tikzset{decorate sep/.style 2 args=
{decorate,decoration={shape backgrounds,shape=circle,shape size=#1,shape sep=#2}}}

\newcommand{\mpartbig}[2]{
\draw [fill] (#1,#2) circle [radius=\vrad]; 
\foreach \x in {-3,-2,-1}{
\draw [line width = \ethick] (#1,#2) -- (#1+\x,#2-3);
\draw [fill] (#1+\x,#2-3) circle [radius=\vrad];
}
\draw [line width = \ethick] (#1,#2) -- (#1+3,#2-3);
\draw [fill] (#1+3,#2-3) circle [radius=\vrad];
\draw [dashed] (#1-0.5,#2-2.5)--(#1+2,#2-2.5) node [midway,below] {$M$};
}
\newcommand{\mpartlittle}[2]{
\draw [fill] (#1,#2) circle [radius=\vrad]; 
\foreach \x in {-3,-2}{
\draw [line width = \ethick] (#1,#2) -- (#1+\x,#2-3);
\draw [fill] (#1+\x,#2-3) circle [radius=\vrad];
}
\draw [line width = \ethick] (#1,#2) -- (#1+2,#2-3);
\draw [fill] (#1+2,#2-3) circle [radius=\vrad];
\draw [dashed] (#1-1,#2-2.5)--(#1+1,#2-2.5) node [midway,below] {$m$};
}
\begin{tikzpicture}[scale=0.7]
\node [above] at (0,0) {$\o$};
\mpartlittle{0}{0}
\mpartbig{-3}{-3}
\mpartlittle{0}{-6}
\end{tikzpicture}
\caption{Constructing $T$ in Example~\ref{ex:p_implies_trivial}.}
\end{figure}
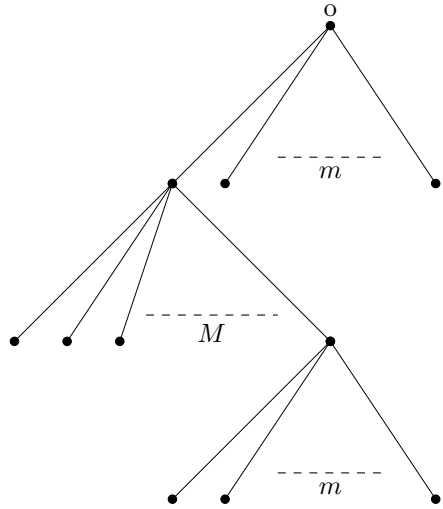

\end{example}
\begin{proof}
Define $f: T \to \C$ as follows:
\[
f(v)=\begin{cases}
0, & \text{ if } |v| \text{ is odd,}\\
(- M)^{-|v|/2}, & \text{ if } |v| \text{ is even.}
\end{cases}
\]
We first verify that $Sf=0$. Indeed, if $|v|$ is even, then if $w$ is
adjacent to $v$, we have $f(w)=0$ and hence $(Sf)(v)=0$. Now, if
$|v|=2k+1$ is odd, then $f(\Par(v))=(-M)^{-k}$ and, if $w \in \Chi(v)$
then $f(w)=(-M)^{-k-1}$. Hence, 
\[
(Sf)(v)= f(\Par(v))+\sum_{w \in \Chi(v)} f(w) = (-M)^{-k} + \sum_{w
  \in \Chi(v)} (-M)^{-k-1} = (-M)^{-k} + M (-M)^{-k-1}=0.
\]
It is clear that $f$ is bounded, which settles the case $p=\infty$.

Now, if $M>1$, let $\log_M(m)+1< p < \infty$. We will show that $f \in
\Lp{T}$. Observe that, for each $k \in \N_0$, there are exactly $(M
m)^k$ vertices $v$ with $|v|=2k$. We then have
\[
\sum_{v \in V} |f(v)|^p = \sum_{|v| \text{ even}} |f(v)|^p =
\sum_{k=0}^\infty \sum_{|v|=2k}  |f(v)|^p = \sum_{k=0}^\infty
\sum_{|v|=2k}  |M|^{-kp} =  \sum_{k=0}^\infty (M m)^k  |M|^{-kp} =
\sum_{k=0}^\infty \left( \frac{m}{M^{p-1}} \right)^k,
\]
which converges since $p-1>\log_M (m)$. Hence $ f \in \Lp{T}$ and
$\ker S$ is nontrivial.
\end{proof}

The idea in the construction of the nontrivial function in the example
above is used below to give a more general theorem.


\begin{theorem}\label{th:p_implies_nontrivial}
Let $T=(V,E)$ be a rooted leafless tree with bounded degree and
consider the shift $S$ on $\Lp{T}$, with $\log_m(M)+1 < p \leq
\infty$, if $m>1$; or $p=\infty$, if $m=1$. Then $\ker(S)$ is
nontrivial.
\end{theorem}
\begin{proof}
First, we will construct a function $f: V \to \C$ such that
$(Sf)(v)=0$ for all $v \in V$. Later, we will show that $f \in \Lp{T}$.

We proceed inductively. Define $f(\o)=1$. For all $v\in V$ with
$|v|=1$, define $f(v)=0$.  It is then clear that $(Sf)(\o)=0$. Now,
assume that, for a fixed $n \in \N$, we have defined $f(v)$ for all
$|v|\leq n$ and that $Sf(v)=0$ for all $v$ with $|v|<n$.

For each $v \in V$ with $|v|=n+1$, we define
$f(v)=\dfrac{-f(\Par^2(v))}{\beta(\Par(v))}$. Then for all $v
\in V$ with $|v|=n$ we have, since $T$ is leafless (and hence $\beta(v)\neq 0$), that
\begin{align*}
    (Sf)(v)= \sum_{u\sim v} f(u)
    &= f(\Par(v)) + \sum_{u \in \Chi(v)} f(u)\\
    &= f(\Par(v)) + \sum_{u \in \Chi(v)} \frac{-f(\Par^2(u))}{\beta(\Par(u))}\\
    &= f(\Par(v)) + \sum_{u \in \Chi(v)} \frac{-f(\Par(v))}{\beta(v)}\\
    &= f(\Par(v)) + \beta(v) \frac{-f(\Par(v))}{\beta(v)}\\
    &=0.
\end{align*}
This completes the induction step; we conclude that $Sf(v)=0$ for all $v \in V$.
Note that $f$ is clearly not identically equal to $0$, as
$f(\o)=1$. Note also that for all odd $n \in \N$, if $v \in V$ and
$|v|=n$, then $f(v)=0$.

It remains to be shown that $f \in \Lp{T}$. Observe that $|f(v)|\leq
1$ for all $v \in V$ and hence $f \in \Linfty{T}$. Assume then that $m>1$ and
$\log_m(M)+1 < p < \infty$. 

Since $T$ has bounded degree, there exists $N \in \N$ such that for
all $|v|\geq 2N$, we have $v \in \calM$. We now prove that
for all $n \in \N$ with $n\geq N$, 
\begin{equation}
\sum_{|u|=2n} \abs{f(u)}^p \leq \frac{M^{n-N}C}{m^{(p-1)(n-N)}}, \label{eq:treesum_nontrivial}
\end{equation}
where
\[
C = \sum_{|u|=2N} \abs{f(u)}^p.
\]
Clearly the desired inequality \eqref{eq:treesum_nontrivial} is true for
$n=N$. Now, assume \eqref{eq:treesum_nontrivial} holds for some fixed $n$ with
$n \geq N$. By the definition of $f$, we obtain
\begin{align*}
    \sum_{|u|=2(n+1)} \abs{f(u)}^p
    &= \sum_{|w|=2n+1} \sum_{u \in \Chi(w)} \abs{f(u)}^p\\
    &= \sum_{|w|=2n+1} \sum_{u \in \Chi(w)} \abs{\frac{-f(\Par^2(u))}{\beta(\Par(u))}}^p \\
    &= \sum_{|w|=2n+1} \beta(w) \abs{\frac{f(\Par(w)}{\beta(w)}}^p \\
    &= \sum_{|w|=2n+1} \frac{1}{\beta(w)^{p-1}} \abs{f(\Par(w))}^p\\
    &\leq \frac{1}{m^{p-1}} \sum_{|w|=2n+1}  \abs{f(\Par(w))}^p,
\end{align*}
where in the last line we use the definition of $m$. We also obtain
\begin{align*}
\sum_{|w|=2n+1}  \abs{f(\Par(w))}^p
&=\sum_{|v|=2n} \sum_{w \in \Chi(v)}  \abs{f(\Par(w))}^p \\
&= \sum_{|v|=2n} \sum_{w \in \Chi(v)}  \abs{f(v)}^p\\
&= \sum_{|v|=2n} \beta(v)  \abs{f(v)}^p\\
&\leq M \sum_{|v|=2n} \abs{f(v)}^p.
\end{align*}
Combining the past two expressions, and using the induction
hypothesis, we get
\[
\sum_{|u|=2(n+1)} \abs{f(u)}^p \leq \frac{M}{m^{p-1}}
\sum_{|v|=2n}\abs{f(v)}^p \leq
\frac{M}{m^{p-1}}\frac{M^{n-N}C}{m^{(p-1)(n-N)}}  =
\frac{M^{n+1-N}C}{m^{(p-1)(n+1-N)}}.
\]
Thus the desired inequality holds for $n+1$ and the induction step is complete.

Observe that, since $f(u)=0$ for all $v \in V$ such that $|u|$ is odd,
to show that $f \in \Lp{T}$ we only need to prove that 
\[
\sum_{n=N}^{\infty}  \sum_{|u|=2n} \abs{f(u)}^p
\]
converges. But
\[
\sum_{n=N}^{\infty}  \sum_{|u|=2n} \abs{f(u)}^p  \leq
\sum_{n=N}^{\infty}  \frac{M^{n-N}C}{m^{(p-1)(n-N)}}
=C\sum_{n=N}^{\infty}  \left(\frac{M}{m^{p-1}}\right)^{n-N},
\]
which converges for $p>1+\log_m M$. Therefore $f \in \Lp{T}$. Having
shown $Sf=0$ while $f\neq0$ and $f \in \Lp{T}$, we conclude that $\ker(S)$ is nontrivial.
\end{proof}

If $m=1$, the above theorem cannot be improved, as the following example shows (observe that the case $M=m=1$ is already covered by Theorem~\ref{th:p_implies_trivial}).

\begin{example}\label{ex:p_implies_nontrivial}
Let $M \in \N$, with $M>1$. Choose a sequence of positive integers
$(t_j)$ such that 
\[
\sum_{j=1}^\infty \frac{t_j}{M^{(j-1)p}}\] 
diverges for every $p \geq 1$. (For example, choose $t_j=2^{(j-1)^2}$ or $t_j=j^{j-1}$.)

Define a tree $T$ as follows. Set the root $\o$ and assign $M$
children to $\o$. Now, to each of these children, assign a path of
length $2t_1 - 1$. At each vertex at the end of each of these paths (we call each such vertex a bifurcation node), assign $M$ children. To each of these
children, attach a path of length $2t_2-1$. At each vertex at the end
of each of these paths, assign $M$ children. To each of these children,
attach a path of length $2t_3-1$. Keep going in this way.

Then, $\ker S$ is trivial in $\Lp{T}$ for $1 \leq p<\infty$.

\begin{figure}[h]
\newcommand\vrad{0.8mm}
\newcommand\ethick{0.1pt}

\usetikzlibrary{decorations.pathreplacing}
\usetikzlibrary{decorations.shapes}

\newcommand{\mpartbig}[2]{
\draw [fill] (#1,#2) circle [radius=\vrad]; 
\foreach \x in {-3,-2,-1}{
\draw [line width = \ethick] (#1,#2) -- (#1+\x,#2-3);
\draw [fill] (#1+\x,#2-3) circle [radius=\vrad];
}
\draw [line width = \ethick] (#1,#2) -- (#1+3,#2-3);
\draw [fill] (#1+3,#2-3) circle [radius=\vrad];
\draw [dashed] (#1-0.5,#2-2.5)--(#1+2,#2-2.5) node [midway,below] {$M$};
}
\newcommand{\stick}[2]{
\foreach \x in {1,2,3,4} {
\draw [fill] (#1,#2-\x) circle [radius=\vrad]; 
}
\draw [line width = \ethick] (#1,#2) -- (#1,#2-2);
\draw [line width = \ethick, dashed] (#1,#2-2) -- (#1,#2-3);
\draw [line width = \ethick] (#1,#2-3) -- (#1,#2-4);
}

\begin{tikzpicture}[scale=0.5]
\mpartbig{0}{0}
\foreach \y in {-3,-2,-1,3}{
\stick{\y}{-3}
}
\draw [decorate,decoration={brace,amplitude=10pt},xshift=-1.5em] (-3,-7) -- (-3,-3) node [midway,xshift=-4.5em] {$2t_1 -1$ edges};

\mpartbig{-3}{-7}
\foreach \y in {-3,-2,-1,3}{
\stick{-3+\y}{-10}
}
\draw [decorate,decoration={brace,amplitude=10pt},xshift=-1.5em] (-6,-14) -- (-6,-10) node [midway,xshift=-4.5em] {$2t_2 -1$ edges};

\end{tikzpicture}
\caption{The tree $T$ in Example~\ref{ex:p_implies_nontrivial}.}
\end{figure}
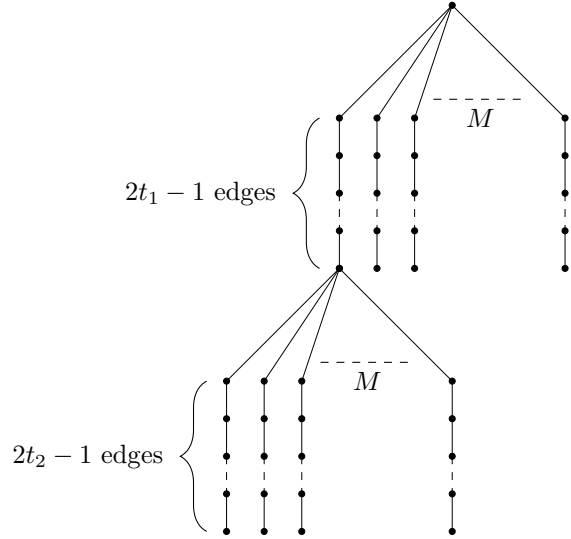

\end{example}
\begin{proof}
Let $f \in \Lp{T}$ with $Sf=0$.

We first label a few nodes, to explain our example. Let $v_0$ be a bifurcation node, and let $v_1, v_2 \dots v_{2s-1}$ be those vertices that comprise a path off of $v_0$; that is, $(v_0,v_1,...v_{2s-1})$ is a path of length $2s-1$ that begins at $v_0$.

Since $(Sf)(v_j)=0$ for each vertex $j=1, 2, \dots 2s-2$, we must have 
\[
f(v_0)=-f(v_2)=f(v_4)= \dots = (-1)^{s+1} f(v_{2s-2})
\]
and
\[
f(v_1)=-f(v_3)=f(v_5)= \dots = (-1)^{s+1} f(v_{2s-1}).
\]

If $f \neq 0$, there must exist a bifurcation node $w^*$ such that $f(w^*)\neq 0$ or
$f(v^*)\neq 0$ for some $v^* \in \Chi(w^*)$, otherwise $f$ would be
identically zero. Let us choose such a $w^*$ closest to the root.

We have two cases:

\begin{itemize}
\item If $f(w^*)\neq 0$, then assume, without loss of generality, that
$f(w^*)=1$. But observe that, if $v_1 \in \Chi^2(w^*)$, then
$f(v_1)=-1$, since $0=(Sf)(\Par(v_1))=f(w^*)+f(v_1)$. Using the observation
above, we obtain that, at the vertex $v_{2t_1-1}$ at the end of the
path, the absolute value of the function is also $1$. But then, if $w_1 \in
\Chi^2(v_{2t_1-1})$, then $|f(w_1)|=1$. Continuing in this manner, it
is clear that there are infinitely many vertices $u$ such that
$|f(u)|=1$, and hence $f \notin \Lp{T}$. Thus $\ker S$ is trivial.

\item If $v_0 \in \Chi(w^*)$ satisfies that $f(v_0)\neq 0$, again, we
  may assume that $f(v_0)=1$. By the observation above, we must have
  that
\[
|f(v_{0})|=|f(v_2)|= |f(v_4)|= \dots = |f(v_{2t_1-2})|=1,
\]
where $v_1, v_2, \dots, v_{2t_1-1}$ are the vertices in the path
attached to $v_0$. Hence there are at least $t_1$ vertices $u$ where $|f(u)|=1$.

Since $(Sf)(v_{2t_1-1})=0$, there must exist a vertex $w_0 \in
\Chi(v_{2t_1-1})$ with $|f(w_0)| \geq 1/M$. Again, by the observation
above, we must have 
\[
|f(w_{0})|=|f(w_2)|= |f(w_4)|= \dots = |f(w_{2t_2-2})|\geq 1/M,
\] 
where $w_1, w_2, \dots, w_{2t_2-1}$ are the vertices in the path
attached to $w_0$. Hence there are at least $t_2$ vertices $u$ where $|f(u)|\geq 1/M$. 

Since $(Sf)(w_{2t_2-1})=0$, there must exist a vertex $x_0 \in
\Chi(v_{2t_2-1})$ with $|f(x_0)| \geq 1/M^2$. Again, by the observation
above, we must have at least $t_3$ vertices $u$ where $|f(u)|\geq 1/M^2$. 

Continuing in this manner, it should be clear that, for each $j \in
\N$, we must have at least $t_j$ vertices $u$ with $|f(u)|\geq
1/M^{j-1}$. But this implies that
\[
\sum_{v\in V} |f(v)|^p \geq \sum_{j=1}^\infty t_j
\frac{1}{M^{(j-1)p}}, 
\]
which diverges. This proves that $f \notin \Lp{T}$ and hence $\ker S$
is trivial. \qedhere
\end{itemize}
\end{proof}

Can Theorem \ref{th:p_implies_nontrivial} be improved for the case $m>1$? We
have not been able to answer the question and we leave it open for
future research. As the above results show, the cases where
$\log_m(M)+1 < p \leq \log_M(m)+1$ seem to depend on how the numbers in the set $\{ \beta(v) : v \in V \}$ are ``distributed'' along the tree, and not only on the numbers $m$ and $M$. We plan to investigate the triviality of $\ker S$ for $L^p(T)$, when $\log_m(M)+1< p \leq \log_M(m)+1$ (which contains the important case $p=2$, unless $M=m$) in the future.

\subsection{The case where $T$ has leaves.} If the tree has leaves, the question of triviality of $\ker S$ seems to be much more complicated. 

For example, suppose the tree $T$, of bounded degree, has at least two vertices of degree $1$, say $u$ and $v$, and the distance $m$ between $u$ and $v$ is even. Let $u=u_0 \sim u_1 \sim u_2 \sim \dots \sim u_{m-1}\sim u_m=v$ be the path between them. Furthermore, if $m \geq 4$, assume that $\deg(u_2)=\deg(u_4)=\dots=\deg(u_{m-2})=2$. If we set $f(u_0)=f(u_4)=f(u_8)=\dots=f(u_{m-2})=1$, $f(u_2)=f(u_6)=\dots=f(u_m)=-1$ and $f(w)=0$ for any other vertex $w$, it is easy to check that $0 \neq f \in \ker S$ and $f\in \Lp{T}$ for $1 \leq p \leq \infty$. 

On the other hand,  we have been able to construct examples of trees where the above conditions do not hold and $\ker S$ is trivial, and also examples where the above conditions do not hold and $\ker S$ is nontrivial. In both cases the triviality of $\ker S$ depends on the value of $p$ and on the ``shape'' of the tree. What is the situation in general? A full characterization seems hard, and we leave the question open.

\section{Eigenvalues of finite graphs with infinite tails}\label{se:spectrum_tails}

The objective of this section is to present a method to find the
set of eigenvalues of certain infinite graphs obtained by attaching to a finite
graph an ``infinite tail''. Some of the results we obtain here have
been obtained by Golinskii in \cite{Golinskii1}: our method has the advantage
of being elementary, while Golinskii's method uses the theory of
Jacobi matrices.

Concretely, say $H$ is a graph with vertices $\{ v_0, v_1, v_2, \dots,
v_n\}$ and $L$ is the graph with vertices $\{ u_k \, : \, k \in \N \}$
and such that $u_k$ is adjacent to $u_j$ if and only if $|k-j|=1$
(what we call an ``infinite tail''). We form the new graph $G_H$ by
taking the union of the vertices of $H$ and $L$, with the same
adjacency relations as before, but adding an edge between the vertex
$v_0$ and $u_1$.

We would like to solve the equation $S f = \lambda
f$ for some complex-valued function $f$ defined on the set $V$ of
vertices of $G_H$. Furthermore, we would like $f$ to be in $\Lp{G_H}$, for $1\leq
p < \infty$. Let us look at the equations obtained from $Sf = \lambda
f$ when we evaluate
$f$ at the vertices $u_j$, for $j \in \N$. For simplicity, denote by
$f_j=f(u_j)$ for $j \in \N$ and $f_0=f(v_0)$. Then
\[
f_{j-1} +f_{j+1}= \lambda f_j
\]
for $j \in \N$. It is well known, and easy to see, that the solutions
to this system are given by
\[
f_j = C_1 b^j + C_2 c^j,
\]
where $b$ and $c$ are the solutions to the equation $t^2-\lambda t
+1=0$ and $C_1, C_2$ are constants to be determined. Since $b c = 1$
we may assume that $|b|\leq 1$ and $c=1/b$. Since we want $f \in
\Lp{G_H}$, we must have $C_2=0$ and either $|b|<1$ or $C_1=0$. Observe
that in the former case, since $\lambda=b + c = b+ 1/b$, we must have $|\lambda| > 2$.

In either case, the problem of finding $\lambda$ reduces to solving the finite system
of equations
\begin{equation}\label{eq:finite}
(Sf)(v_j) = \lambda f(v_j)
\end{equation}
for $j=0, 1, 2, \dots, n$ and such that $f(v_0)=C_1$ and $f(u_1)=C_1 b$. 

If $C_1=0$, then $f_j=0$ for all $j \in \N_0$ and hence the system of equations can be
seen as finding an eigenvalue for the adjacency matrix of $H$ such
that the corresponding eigenvector has value $0$ at $v_0$. 

If $C_1\neq 0$, dividing $f$ by a constant, we may always assume that
$C_1=1$. In this case, the system of equations does not correspond to
the eigenvalue equations for the adjacency matrix of $H$, since the
vertex $u_1$, which is not in $H$, is also adjacent to
$v_0$. Nevertheless, in several cases the finite system of equations
can be solved. We show a few examples where we can achieve this.

\subsection{The kite with infinite tail}

Let us consider a cycle graph with an infinite tail attached. That is,
$H$ is the $(n+1)$-cycle (with $n\geq 2$), which consists of the
vertices $v_0, v_1, v_2, \dots, v_n$  and the edges given by the
relations $v_j \sim v_{j+1}$ for $j=0, 1, 2 \dots, n-1$ and $v_n \sim
v_0$. The infinite graph $G_H$ is obtained by attaching an infinite
tail to the vertex $v_0$. We call this graph the {\em kite with an infinite
  tail}.
  
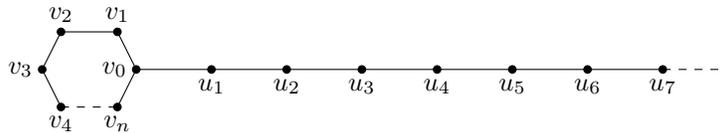
\begin{figure}[htbp]
\begin{tikzpicture} 

\draw (0,2) --(7,2);
\draw [dashed] (7, 2) -- (8,2);
\draw [fill=black] (0, 2) circle
[radius=0.05];
\node [left] at (0, 2) {$v_0$};
\foreach \x in {1,...,7}
{
\draw [fill=black] (\x, 2) circle
[radius=0.05];
\node [below] at (\x, 2) {${u_\x}$}; 
}

\draw (0,2) -- (-0.25, 2.5);
\draw [fill=black](-0.25, 2.5) circle
[radius=0.05];
\node [above] at (-0.25,2.5) {$v_1$};
\draw (-0.25, 2.5) -- (-1, 2.5);
\draw [fill=black](-1, 2.5) circle
[radius=0.05];
\node [above] at (-1,2.5) {$v_2$};
\draw (-1, 2.5) -- (-1.25, 2);
\draw [fill=black] (-1.25, 2) circle
[radius=0.05];
\node [left] at (-1.25, 2) {$v_3$};
\draw (-1.25, 2) -- (-1, 1.5);
\node [below] at (-1, 1.5) {$v_4$};
\draw [fill=black] (-1, 1.5) circle
[radius=0.05];
\draw [dashed] (-1, 1.5) -- (-0.25, 1.5);
\node [below] at (-0.25, 1.5) {$v_{n}$}; 
\draw [fill=black] (-0.25, 1.5) circle
[radius = 0.05];
\draw(-0.25, 1.5) -- (0,2); 
\end{tikzpicture}
\caption{The kite with an infinite tail.}
\end{figure} 
We will use the argument outlined above to obtain the set of
eigenvalues of the kite.

Let us deal with the case $C_1=0$ first. It is well-known (see, for
example,  \cite[p.~8-9]{BrHa}) that the eigenvalues of the adjacency
matrix of the $(n+1)$-cycle are the numbers $\{ 2 \cos(2 \pi j/(n+1)) \ : j=0,
1, 2, \dots, [(n+1)/2] \}$ (here, as usual, $[x]$ denotes the floor of $x$).

The corresponding eigenspaces are obtained in the following manner. Let
$\omega=\exp(i 2\pi /(n+1))$. Then the eigenspace corresponding to the
eigenvalue $2 \cos(2 \pi j/(n+1))$ is the linear span of the vectors
\[
(1, \omega^j, \omega^{2j}, \omega^{3j}, \dots, \omega^{nj}) \qquad
\text{ and } \qquad (1, \omega^{-j}, \omega^{-2j}, \omega^{-3j}, \dots, \omega^{-nj}).
\]
Observe that if $n$ is even, then these eigenspaces are of dimension
$2$ if $j=1, 2, \dots, n/2$, and of dimension $1$ if $j=0$. If $n$
is odd, these eigenspaces are of dimension $2$ if $j=1, 2, \dots,
(n-1)/2$, and of dimension $1$ if $j=0$ or $j=(n+1)/2$.  

Each nonzero vector in the eigenspaces of dimension $1$ does not have a zero entry
at the position corresponding to $v_0$. But clearly, we can choose an
eigenvector in each of the eigenspaces of dimension $2$ such that the
value at $v_0$ is zero. Hence, the eigenvalues for the kite with
infinite tail, in this case, are the numbers
\[
\{ 2 \cos(2 \pi j/(n+1)) \ : j=1, 2, \dots, [n/2] \}.
\]

Now, let us deal with the case $C_1 \neq 0$. As we mentioned above, we
can always assume that $C_1 =1$. If we set $x_j:=f(v_j)$, the system of equations
\eqref{eq:finite} becomes:
\begin{align*}
1 + x_2 &= \lambda x_1 \\
x_2 + x_3 &= \lambda x_2 \\
x_3 + x_4 &= \lambda x_3 \\
\vdots &  \vdots \\
x_{n-1}+1 &= \lambda x_n \\
x_1+x_n+b &= \lambda.
\end{align*}
The first $n$ equations can be rewritten, in matrix form, as $A X=C$, where
\[
A= \begin{pmatrix}
\lambda & -1 & 0 & 0 & \dots & 0 & 0 & 0 \\
-1 & \lambda & -1 & 0 & \dots & 0 & 0 & 0 \\
0 & -1 & \lambda & -1 & \dots & 0 & 0 & 0 \\
\vdots & \vdots & \vdots & \vdots & & \vdots & \vdots & \vdots  \\
0 & 0 & 0 & 0 & \dots  & \lambda & -1 & 0 \\
0 & 0 & 0 & 0 & \dots  & -1 & \lambda & -1 \\
0 & 0 & 0 & 0 & \dots & 0  & -1 & \lambda
\end{pmatrix}
\quad \text{ and } \quad
C=\begin{pmatrix} 1 \\ 0 \\ 0 \\ \vdots  \\ 0 \\ 0 \\ 1 
\end{pmatrix}
\]

It is well known (e.g., \cite{BrHa}) that the eigenvalues of the matrix
\[
\begin{pmatrix}
0& 1 & 0 & 0 & \dots & 0 & 0 & 0 \\
1 & 0 & 1 & 0 & \dots & 0 & 0 & 0 \\
0 & 1 & 0 & 1 & \dots & 0 & 0 & 0 \\
\vdots & \vdots & \vdots & \vdots & & \vdots & \vdots & \vdots  \\
0 & 0 & 0 & 0 & \dots  & 0 & 1 & 0 \\
0 & 0 & 0 & 0 & \dots  & 1 & 0 & 1 \\
0 & 0 & 0 & 0 & \dots & 0  & 1 & 0
\end{pmatrix}
\]
are all in the interval $(-2,2)$. Since $|\lambda|>2$, it follows that
$A$ is invertible and there is a unique solution to the equation $A X
= C$.

A tedious, but straightforward, computation shows that the solution of $AX=C$ is given by the expressions
\[
x_k=\frac{b^k+b^{n-k+1}}{1+b^{n+1}},
\]
for $k=1, 2, \dots n$. The solutions also need to satisfy the equation
\[
x_1+x_n=\lambda - b.
\]
Hence we obtain
\[
\frac{b+b^{n}}{1+b^{n+1}} + \frac{b+b^{n}}{1+b^{n+1}} = \frac{1}{b},
\]
which simplifies to
\[
b^{n+1}+2b^2-1=0.
\]
Consider the polynomials $p_n(x)=x^{n+1}+2x^2-1$. The case $n=2$ is
special. In that case, $p_2(x)=x^3+2x^2-1=(x+1)(x^2+x-1)$ has roots
\[
-1, \quad \frac{-1-\sqrt{5}}{2} \approx -1.618, \quad \text{ and } \frac{-1+\sqrt{5}}{2}\approx 0.618, 
\]
Since only one of the roots is of modulus less than 1, this choice of $b$ gives the eigenvalue $\lambda=\sqrt{5}$.

If $n$ is odd, it is easy to check that $p_n$ is decreasing in the
interval $(-\infty,0)$ and  increasing in the interval
$(0,\infty)$. Since $p_n(0)=-1<0$ and $p_n(1)=p_n(-1)=2>0$, it follows
that there is exactly one root in the interval $(-1,0)$ and one in the
interval $(0,1)$. Hence, if $n$ is odd, in this case there will be two
eigenvalues (one the negative of the other) for the kite.

Now, if $n$ is even, $n\geq 4$, define
\[
\alpha_n=\left(\frac{-4}{n+1}\right)^{1/(n-1)}
\]
It is clear that $-1 < \alpha_n < 0$ and it is easy to check that
$p_n$ is increasing on the intervals $(-\infty, \alpha_n)$ and
$(0,\infty)$ and decreasing on the interval $(\alpha_n,0)$. Since
$p_n(-1)=0$, $p_n(0)=-1$ and $p_n(1)=2$, it follows that there is
exactly one root in the interval $(-1,0)$ and one in the interval
$(0,1)$.

In both cases, it turns out there are exactly two roots of $p_n$ (for
$n\neq 2$) in the interval $(-1,1)$, denoted by $x_1(n)$ and $x_2(n)$ and such that by $-1< x_2(n) < 0 <x_1(n)<1$.

Putting these results together with the computations for the case $C_1=0$, we obtain that
the eigenvalues of the kite $G_H$, in $\Lp{G_H}$, are 
\[
\{ -1, \sqrt{5} \}
\]
if $n=2$, and
\[
\{ 2 \cos(2 \pi j/(n+1)) \ : j=1, 2, \dots, [n/2] \} \bigcup \{ x_1(n)+
\frac{1}{x_1(n)},  x_2(n)+ \frac{1}{x_2(n)} \}
\]
if $n \geq 3$.

(For a complete spectral analysis of this graph in the case $p=2$, see \cite[Proposition 1.4]{Golinskii1}.)

\subsection{The fly-swatter with infinite tail}

Let us consider the fly-swatter. That is, $H$ is the complete graph
$K_{n+1}$, with $n\geq 2$, and $G_H$ is obtained by attaching an infinite tail to one
of the vertices (say the vertex $v_0$). We call this graph the {\em
  fly-swatter}. We will use the argument outlined above to obtain the
set of eigenvalues of the fly-swatter.

\begin{figure} [htbp]
  \begin{tikzpicture} 

\draw (0,2) --(7,2);
\draw [dashed] (7, 2) -- (8,2);
\draw [fill=black] (0, 2) circle
[radius=0.05];
\node [above right] at (0, 2) {$v_0$};
\foreach \x in {1,...,7}
{
\draw [fill=black] (\x, 2) circle
[radius=0.05];
\node [below] at (\x, 2) {${u_\x}$}; 
}

\draw (0,2) -- (-0.25, 2.5);
\draw [fill=black](-0.25, 2.5) circle
[radius=0.05];
\node [above] at (-0.25,2.5) {$v_1$};
\draw (-0.25, 2.5) -- (-1, 2.5);
\draw [fill=black](-1, 2.5) circle
[radius=0.05];
\node [above] at (-1,2.5) {$v_2$};
\draw (-1, 2.5) -- (-1.25, 2);
\draw [fill=black] (-1.25, 2) circle
[radius=0.05];
\node [left] at (-1.25, 2) {$v_3$};
\draw (-1.25, 2) -- (-1, 1.5);
\node [below] at (-1, 1.5) {$v_4$};
\draw [fill=black] (-1, 1.5) circle
[radius=0.05];
\draw [dashed] (-1, 1.5) -- (-0.25, 1.5);
\node [below] at (-0.25, 1.5) {$v_{n}$}; 
\draw [fill=black] (-0.25, 1.5) circle
[radius = 0.05];
\draw(-0.25, 1.5) -- (0,2); 

\draw (0,2) -- (-1, 2.5);
\draw (0,2) -- (-1.25, 2);
\draw (0,2) -- (-1, 1.5);
\draw (0,2) -- (-0.25, 1.5);

\draw (-0.25, 2.5) -- (-1.25, 2);
\draw (-0.25, 2.5) -- (-1, 1.5);
\draw (-0.25, 2.5) -- (-0.25, 1.5);

\draw (-1, 2.5) -- (-1, 1.5);
\draw (-1, 2.5) -- (-0.25,1.5);

\draw (-1.25, 2) -- (-0.25,1.5);
\end{tikzpicture}
\caption{The fly-swatter.}
\end{figure}
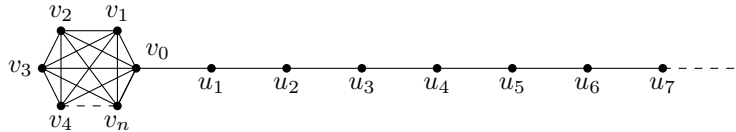

Again, let us deal with the case $C_1=0$ first. It is well-known (see, for
example,  \cite[p.~8]{BrHa}) that the set of eigenvalues of the adjacency
matrix of the complete graph on $(n+1)$ vertices is $\{ -1,
n\}$. The eigenspace corresponding to the eigenvalue $n$ is spanned
by the vector $(1,1, \dots, 1)$ and the eigenspace corresponding to
the eigenvalue $-1$ is the set
\[
\{ (v_0, v_1, v_2, \dots, v_n) \in \C^{n+1} \, : \, v_0 + v_1 + v_2 + \dots + v_n =0 \}.
\]
Since an eigenvector with $v_0=0$ can be chosen only in the case where
the eigenvalue is $-1$, this shows that in the case $C_1=0$, the only
eigenvalue of the fly-swatter is $-1$.

Now, if $C_1=1$, and we set $x_j:=f(v_j)$, the system of equations
\eqref{eq:finite} becomes:
\begin{align*}
1+ x_2 + x_3+ x_4 + \dots + x_n ={} & \lambda x_1 \\
1+ x_1 + x_3+ x_4 + \dots + x_n ={} & \lambda x_2 \\
1+ x_1 + x_2+ x_4 + \dots + x_n ={} & \lambda x_2 \\
\vdots  & \vdots \\
1+ x_1 + x_2+ x_3 + \dots + x_{n-1} ={} & \lambda x_n \\
x_1 + x_2+ x_4 + \dots + x_n + b ={} & \lambda \\
\end{align*}

The first $n$ equations can be rewritten, in matrix form, as $A X=C$, where
\[
A= \begin{pmatrix}
-\lambda & 1 & 1 & 1 & \dots & 1 & 1 & 1 \\
1 & -\lambda & 1 & 1 & \dots & 1 & 1 & 1 \\
1 & 1 & -\lambda & 1 & \dots & 1 & 1 & 1 \\
\vdots & \vdots & \vdots & \vdots & & \vdots & \vdots & \vdots  \\
1 & 1 & 1 & 1 & \dots  & -\lambda & 1 & 1 \\
1 & 1 & 1 & 1 & \dots  & 1 & - \lambda & 1 \\
1 & 1 & 1 & 1 & \dots & 1  & 1 & -\lambda
\end{pmatrix}
\quad \text{ and } \quad
C=\begin{pmatrix} -1 \\ -1 \\ -1 \\ \vdots  \\ -1 \\ -1 \\ -1 
\end{pmatrix}
\]
and the solution $(x_1, x_2, x_3, \dots, x_n)$ must satisfy the last equation:
\[
x_1 + x_2+ x_4 + \dots + x_n = \lambda -b.
\]

Now, the matrix $A$ is invertible if $\lambda \neq -1$ and $\lambda
\neq n-1$, since if we set $\lambda=0$, then the matrix $A$ is the adjacency
matrix for the complete graph on $n$ vertices. Since we are dealing
with the case $C_1=1$, we have that $\lambda >2$ and hence we only
need to discard the case $\lambda=n-1$.  Clearly, however, $C \in \Ker
A$ if $\lambda=n-1$, and $A$ is selfadjoint, so we must have that the range of
$A$ is orthogonal to the kernel of $A$. Hence there is no solution in
this case.

Now, for every other value of $\lambda$ there is a unique solution to
$A X = C$, and it can be easily checked that it is given by
\[
x_1=x_2=\dots = x_n= \frac{1}{\lambda-(n-1)}.
\]
Now, this solution should satisfy the equation
\[
x_1 + x_2+ x_4 + \dots + x_n + b = \lambda,
\]
from which we obtain $\frac{n}{\lambda-(n-1)}+b =\lambda
$. Simplifying and using the fact that $\lambda=b+1/b$ we obtain
\[
(n-1) b^2 + (n-1) b - 1 =0.
\]
The solutions to this equation are
\[
b=-\frac{1}{2}\pm \frac{1}{2} \sqrt{\frac{n+3}{n-1}}.
\]
It is easy to check that $0<-\frac{1}{2}+\frac{1}{2}
\sqrt{\frac{n+3}{n-1}}<1$ and  $-2<-\frac{1}{2}-\frac{1}{2}
\sqrt{\frac{n+3}{n-1}}<-1$. Hence, $b=-\frac{1}{2}+\frac{1}{2}
\sqrt{\frac{n+3}{n-1}}$ and
\[
\lambda=b+1/b= \frac{1}{2}\left(-1+ \sqrt{\frac{n+3}{n-1}}\right) +
\frac{2}{-1+ \sqrt{\frac{n+3}{n-1}}}
\]

Putting these results together with the computations for the case $C_1=0$, we obtain that
the eigenvalues of the fly-swatter $G_H$, in $\Lp{G_H}$ are
\[
\left\{  -1, \frac{1}{2}\left(-1+ \sqrt{\frac{n+3}{n-1}}\right) + \frac{2}{-1+ \sqrt{\frac{n+3}{n-1}}} \right\}.
\]

(For the spectral analysis of the complete bipartite graph with tail in the case $p=2$, see \cite[Example 5.6]{Golinskii1}.)

\subsection{The comb with an infinite tail}

Let us consider the comb with an infinite tail. That is, $H$ is the
graph with vertices $\{ v_1, v_2, \dots, v_n, w_1, w_2, \dots w_n \}$
and such that $v_j \sim v_{j+1}$ for $j=1, 2, \dots, n-1$, and
$v_j\sim w_j$ for $j=1, 2, \dots, n$, where $n \geq 2$; and  $G_H$ is
obtained by attaching an infinite tail to the vertex $v_n$.

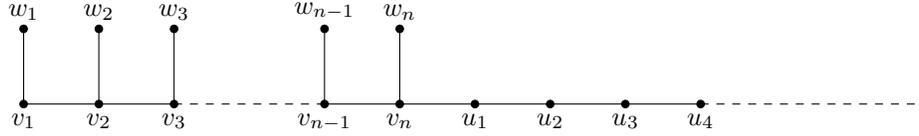
\begin{figure}[htbp]
\begin{tikzpicture} 

\draw (1,2) --(3,2);
\foreach \x in {1,...,3}
{
\draw [fill=black] (\x, 2) circle
[radius=0.05];
\node [below] at (\x, 2) {$v_{\x}$}; 
}
\draw (1,2) -- (1,3); 
\node [above] at (1,3) {$w_1$};
\draw [fill=black] (1,3) circle
[radius=0.05];
\draw [fill=black] (2,3) circle
[radius=0.05];
\node [above] at (2,3) {$w_2$};
\draw (2,2) -- (2,3);
\draw (3,2) -- (3,3);
\draw [fill=black] (3,3) circle
[radius=0.05];
\node [above] at (3,3) {$w_3$};
\draw [dashed] (3,2) -- (5,2);

\draw [fill=black] (5,2) circle
[radius=0.05];
\node [below] at (5,2) {$v_{n-1}$};
\draw (5,2) -- (5,3);
\node [above] at (5,3) {$w_{n-1}$};
\draw [fill=black] (5,3) circle
[radius=0.05];
\draw (5,2) -- (6,2);
\draw [fill=black] (6,2) circle
[radius=0.05];
\node [below] at (6,2) {$v_n$};
\draw (6,2) -- (6,3);
\node [above] at (6,3) {$w_n$};
\draw [fill=black] (6,3) circle
[radius=0.05];

\draw (6,2) -- (7,2);
\draw [fill=black] (7,2) circle
[radius=0.05];
\node [below] at (7,2) {$u_1$};

\draw (7,2) -- (8,2);
\draw [fill=black] (8,2) circle
[radius=0.05];
\node [below] at (8,2) {$u_2$};

\draw (8,2) -- (9,2);
\draw [fill=black] (9,2) circle
[radius=0.05];
\node [below] at (9,2) {$u_3$};

\draw (9,2) -- (10,2);
\draw [fill=black](10,2) circle
[radius=0.05];
\node [below] at (10,2) {$u_4$};

\draw [dashed] (10,2) -- (13,2);
\end{tikzpicture}
\caption{The comb with an infinite tail.}
\end{figure}

We call this graph $G_H$ the {\em comb with an infinite tail}
(or just the comb, for short). We will use the argument outlined above
to obtain the set of eigenvalues of the comb.

Let us show that there are no eigenvalues in the case $C_1=0$. Assume
$f$ is an eigenvector with $f(v_n)=0$ and set $y_j=f(v_j)$ and
$z_j=f(w_j)$ for $j=1, 2, \dots, n$. The adjacency matrix of $S$  for
the graph $H$ is the following $2n \times 2n$ matrix:
\[
\left(
\begin{array}{cccccc|cccccc}
0 & 1 & 0 &\dots & 0 & 0 & 1 & 0 & 0 & \dots & 0 & 0 \\
1 & 0 & 1 &\dots & 0 & 0 & 0 & 1 & 0 & \dots & 0 & 0 \\
0 & 1 & 0 &\dots & 0 & 0 & 0 & 0 & 1 & \dots & 0 & 0 \\
\vdots & \vdots & \vdots & & \vdots & \vdots & \vdots & \vdots  & \vdots & & \vdots & \vdots \\
0 & 0 & 0 &\dots & 0 & 1 & 0 & 0  & 0  &\dots & 1 & 0\\
0 & 0 & 0 &\dots & 1 & 0 & 0 & 0  & 0  &\dots & 0& 1 \\
\hline
1 & 0 & 0 &\dots & 0 & 0 & 0 & 0 & 0 & \dots & 0 & 0\\
0 & 1 & 0 &\dots & 0 & 0 & 0 & 0 & 0 & \dots & 0 & 0\\
0 & 0 & 1 &\dots & 0 & 0 & 0 & 0 & 0 & \dots & 0 & 0\\
\vdots & \vdots & \vdots & & \vdots & \vdots &  \vdots & \vdots &  \vdots & & \vdots & \vdots \\
0 & 0 & 0 &\dots & 1 & 0 & 0 & 0  & 0  &\dots & 0 & 0 \\
0 & 0 & 0 &\dots & 0 & 1 & 0 & 0  & 0  &\dots & 0 & 0 
\end{array}
\right).
\]
If we set $T$ to be the upper-left $n \times n$ corner of the matrix
above, $Y=(y_1, y_2, \dots, y_n)$ and $Z=(z_1, z_2, \dots, z_n)$, the
eigenvalue problem $Sf=\lambda f$ becomes $TY+Z=\lambda Y$ and
$Y=\lambda Z$. If $\lambda=0$ then $Y=0$ and hence $Z=0$, hence
$\lambda \neq 0$.

Now, since we want $y_n=0$, it follows that $z_n=0$. The last entry of
the equation $TY+Z = \lambda Y$ is $y_{n-1}+z_n=\lambda y_n$ which
implies that $y_{n-1}=0$ and hence that $z_{n-1}=0$. The penultimate
entry of $TY+Z = \lambda Y$ is $y_{n-2}+y_n + z_{n-1}=\lambda y_{n-1}$
which implies that $y_{n-2}=0$ and hence $z_{n-2}=0$. Proceeding in
this manner, we obtain that $Y=0$ and $Z=0$ and hence there are no
eigenvalues in the case $C_1=0$.

Now let us deal with the case $C_1=1$. In this case, setting
$x_j=f(v_j)$ and $a_j=f(w_j)$ for $j=1, 2, \dots, n$, the system of
equations \eqref{eq:finite} becomes
\begin{align*}
a_1+ x_2 ={} & \lambda x_1 \\
a_{j}+ x_{j-1}+ x_{j+1} ={} & \lambda x_{j} \qquad \text{ for all } j=2, 3, \dots, n-2 \\
a_{n-1}+ x_{n-2}+ 1 ={} & \lambda x_{n-1}\\
a_n+ x_{n-1}+ b ={} & \lambda  \\
x_j ={} & \lambda a_j \qquad \text{ for all } j=1, 2, \dots, n,
\end{align*}
where, as mentioned before, $x_n=1$.

First, observe that if $\lambda=0$, this would contradict the fact
that $1=x_n=\lambda a_n$, and hence we may assume that $\lambda \neq
0$. Since we have $a_j=\frac{x_j}{\lambda}$ for all $j=1, 2, \dots, n$
we can rewrite the above system as
\begin{eqnarray*}
\frac{1}{\lambda} x_1+ x_2 & ={} & \lambda x_1 \\
\frac{1}{\lambda} x_{j}+ x_{j-1}+ x_{j+1} & ={} & \lambda x_{j} \qquad \text{ for all } j=2, 3, \dots, n-2 \\
\frac{1}{\lambda} x_{n-1}+ x_{n-2}+ 1 & ={} & \lambda x_{n-1}\\
\frac{1}{\lambda} + x_{n-1}+ b & ={} & \lambda.
\end{eqnarray*}

Thus we need to find a solution to the system $AX=C$ where
\[
A= \begin{pmatrix}
\lambda -\frac{1}{\lambda} & -1 & 0 & 0 & \dots & 0  & 0 \\
-1 & \lambda -\frac{1}{\lambda}& -1 & 0 & \dots & 0 & 0  \\
0 & -1 & \lambda - \frac{1}{\lambda} & -1 & \dots & 0 & 0 \\
\vdots & \vdots & \vdots & \vdots & & \vdots & \vdots \\
0 & 0 & 0 & 0 & \dots  & \lambda-\frac{1}{\lambda} & -1 \\
0 & 0 & 0 & 0 & \dots  & -1 &  \lambda-\frac{1}{\lambda} \\
0 & 0 & 0 & 0 & \dots & 0 & 1 
\end{pmatrix}
\quad \text{ and } \quad
C=\begin{pmatrix} 0 \\ 0 \\ 0 \\ \vdots  \\ 0 \\ 1 \\ \lambda-\frac{1}{\lambda}-b 
\end{pmatrix},
\]
where $A$ is a $n\times (n-1)$ matrix. Clearly, the columns of $A$ are
linearly independent, and hence the rank of $A$ is $n-1$. Therefore
there is a solution to $AX=C$ if and only if the augmented matrix $(A
| C)$ has rank $n-1$. This occurs if and only if
\[
\det\begin{pmatrix}
\lambda -\frac{1}{\lambda} & -1 & 0 & 0 & \dots & 0  & 0 & 0\\
-1 & \lambda -\frac{1}{\lambda}& -1 & 0 & \dots & 0 & 0  & 0\\
0 & -1 & \lambda - \frac{1}{\lambda} & -1 & \dots & 0 & 0 & 0\\
\vdots & \vdots & \vdots & \vdots & & \vdots & \vdots & \vdots  \\
0 & 0 & 0 & 0 & \dots  & \lambda-\frac{1}{\lambda} & -1 & 0\\
0 & 0 & 0 & 0 & \dots  & -1 &  \lambda-\frac{1}{\lambda} & 1\\
0 & 0 & 0 & 0 & \dots & 0 & 1 & \lambda-\frac{1}{\lambda}-b 
\end{pmatrix}=0
\]
Since $\lambda=b+\frac{1}{b}$, we have
$\lambda-\frac{1}{\lambda}=\frac{b^4+b^2+1}{b(b^2+1)}$ and $\lambda-\frac{1}{\lambda}
-b=\frac{1}{b(b^2+1)}$ and the above determinant can be thought of as
a function of $b$. Denote the determinant of the $n \times n$ matrix
by $P_n(b)$.

An explicit expression for the function $P_n$ can be obtained
recursively. Indeed, set $P_1(b):=\lambda-\frac{1}{\lambda} -b$ and compute
$P_2$ directly to obtain
\[
P_2(b)=\frac{-b^6-b^4+1}{(b(b^2+1))^2}.
\]
Then, it easily follows that $P_n(b)= (\lambda-\frac{1}{\lambda}) P_{n-1}(b) - P_{n-2}(b)$, for $n\geq 3$.

We would like to find out if there are solutions $b$ of the equation
$P_n(b)=0$ of modulus less than one. It can be shown by induction
that, for $n\geq 2$ we have that $h_n(b):=(b(b^2+1))^n P_n(b)$ is an
even polynomial of degree $4n-2$, the coefficient of the term of
degree $4n-2$ is $-1$ and constant term is equal to $1$. In fact, it
can be easily seen that, if we set $h_1(b)=1$ and $h_2(b)=-b^6-b^4+1$,
the polynomial $h_n$ is given by
\[
h_n(b)= (b^4+b^2+1) h_{n-1}(b) - b^2(b^2+1)^2 h_{n-2}(b),
\]
for $n \geq 3$.

To obtain eigenvalues of the comb, it therefore suffices to compute the
roots of $h_n$ of modulus less than one. Since $h_n(0)=1$ and the
coefficient of the term of degree $4n-2$ is $-1$, it follows that $h_n$
has at least one positive root, for each $n\geq 2$.

It is straightforward to check that $h_2(b)=-b^6-b^4+1$ has exactly
one positive root, which happens to be less than $1$.  Let us denote
by $\alpha_n$ the smallest positive root of $h_n$, for $n\geq 2$. We
will show that the sequence $(\alpha_n)$ is decreasing. First, observe
that
\[
h_{3}(\alpha_2)= (\alpha_2^4+\alpha_2^2+1) h_{2}(\alpha_2) -
\alpha_2^2(\alpha_2^2+1)^2 h_{1}(\alpha_2)=-
\alpha_2^2(\alpha_2^2+1)^2.
\]
Hence $h_{3}(\alpha_2)<0$ and since $h_3(0)=1>0$, it follows that
$\alpha_3< \alpha_2$. Now, assume that $\alpha_2>\alpha_3>\dots >
\alpha_{k-1}> \alpha_k$. Then,
\[
h_{k+1}(\alpha_k)= (\alpha_k^4+\alpha_k^2+1) h_{k}(\alpha_k) -
\alpha_k^2(\alpha_k^2+1)^2 h_{k-1}(\alpha_k)=-
\alpha_k^2(\alpha_k^2+1)^2 h_{k-1}(\alpha_k).
\]
Since $h_{k-1}(\alpha_k)>0$ (otherwise, since $h_{k-1}(0)=1$, it would
follow that $\alpha_{k-1} \leq \alpha_k$) we have that
$h_{k+1}(\alpha_k) <0$ and hence $\alpha_{k+1} < \alpha_k$. Since the
sequence $(\alpha_n)$ is decreasing and $\alpha_2 < 1$, it follows
that $h_n$ has at least two roots in the interval $(-1,1)$ and hence
the comb has at least two eigenvalues (one the negative of the other).

One can check using a computer algebra system or graphical software,
that for $2 \leq n \leq 6$ there are exactly two roots of $h_n$ inside
the interval $(-1,1)$, for $7 \leq n \leq 10$ there are exactly four
roots of $h_n$ inside the interval $(-1,1)$, for $11 \leq n\leq 14$
there are exactly six roots of $h_n$ inside the interval $(-1,1)$, for
$15 \leq n \leq 19$ there are exactly eight roots of $h_n$ inside the
interval $(-1,1)$, for $20 \leq n \leq 23$ there are exactly ten roots
of $h_n$ inside the interval $(-1,1)$, and for $n\geq 24$ there are at
least twelve roots of $h_n$ inside the interval $(-1,1)$.

This will correspond to two eigenvalues of the comb for $2 \leq n \leq
6$, four eigenvalues of the comb for $7 \leq n \leq 10$, to six
eigenvalues of the comb for $11 \leq n\leq 14$, to eight eigenvalues
of the comb for $15 \leq n \leq 19$, to ten eigenvalues of the comb
for $20 \leq n \leq 23$, and to at least twelve eigenvalues of the
comb for $n\geq 24$. Can one compute the number of eigenvalues for any
given $n$? We leave that question open for future research.

One can also check, with the help of a computer, that for $2\leq n \leq
30$ all eigenvalues of the comb are contained in the set
$[-1-\sqrt{2},-2] \cup [2,1+\sqrt{2}]$. Does all of the above hold for
all values of $n$?

Note added: Professor L.~Golinskii has shown \cite{Golinskii3} that all eigenvalues of the comb with an infinite
  tail attached are simple and are contained in  the set $(-1-\sqrt{2},-2) \cup
  (2,1+\sqrt{2})$. Furthermore, he has proved that the largest positive eigenvalue of the $n$-comb
  tends to $1+\sqrt{2}$ as $n\to \infty$. Also, he has obtained an
  explicit expression for the number of eigenvalues, in terms of $n$,
  for the $n$-comb.

\subsection{Full spectrum of finite graphs with an infinite tail}

We should point out that the above results actually give the full
spectrum of each of the graphs with an infinite tail attached. Indeed,
in \cite{Duren} Duren finds the spectrum of Toeplitz operators with
finitely many nonzero diagonals. Applying those results to the case of
the Toeplitz operator with matrix
\[
T=\begin{pmatrix}
0 & 1 & 0 & 0 & 0 & \dots \\
1 & 0 & 1 & 0 & 0 & \dots \\
0 & 1 & 0 & 1 & 0 & \dots \\
0 & 0 & 1 & 0 & 1 &  \dots \\
0 & 0 & 0 & 1 & 0 &  \dots \\
\vdots & \vdots & \vdots & \vdots & \vdots & \ddots
\end{pmatrix}
\]
we obtain that the spectrum of $T$ in $\ell^p$ ($1\leq p \leq \infty$)
is the interval $[-2,2]$. In fact, it can be deduced from Duren's
result that the essential spectrum (see, for example, \cite{Muller} for the
definition of the essential spectrum) of $T$ is the interval
$[-2,2]$. 

Observe that $T$ is the adjacency matrix of the infinite tail. Since
all the graphs in this section can be seen as perturbations of the
infinite tail by an operator of finite rank, it follows that the
essential spectrum of each finite graph with an infinite tail attached
(in $\ell^p$, with $1 \leq p \leq \infty$) is the interval
$[-2,2]$. Hence, the spectrum (in $\ell^p$) of each finite graph with an infinite
tail attached is the union of the interval $[-2,2]$ and the set of
eigenvalues (see \cite{Muller} for details on the essential spectrum and
point spectrum).


\section{Spectrum of the infinite comb}

Let us consider the ``infinite comb''. That is, $G$ is the graph
formed with the vertices $\{ v_j, w_j \, : \, j \in \N \}$ and the
edges given by the relations $v_j \sim w_j$ for all $j \in \N$, $v_j
\sim  v_{j+1}$ for all $j \in \N$. First, let us observe that there
are no eigenvalues for $S$ in this case. Indeed, assume $Sf=\lambda f$
and set $x_j=f(v_j)$ and $a_j=f(w_j)$ for $j \in \N$. We obtain the
equations
\begin{align*}
a_1+ x_2 ={} & \lambda x_1, \\
a_{j}+ x_{j-1}+ x_{j+1} ={} & \lambda x_{j}, \qquad \text{ for all } j\geq 2, \\
x_j ={} & \lambda a_j, \qquad \text{ for all } j\in \N.
\end{align*}
Clearly, if $\lambda=0$, then $f=0$ so we may assume $\lambda \neq 0$. Then, the above equations can be written as
\begin{equation}\label{eq:infinite}
\begin{split}
x_2 & =  \left(\lambda -\frac{1}{\lambda}\right) x_1, \\
x_{j-1}+ x_{j+1} & =  \left(\lambda -\frac{1}{\lambda}\right) x_{j}, \qquad \text{ for all } j\geq 2.
\end{split}
\end{equation}
But, as is well-known, the infinite matrix
\[
T:=\begin{pmatrix}
0& 1 & 0 &  \dots & 0 & 0 & 0 & \cdots \\
1 & 0 & 1 &  \dots & 0 & 0 & 0 & \cdots \\
0 & 1 & 0 &  \dots & 0 & 0 & 0 & \cdots \\
\vdots & \vdots & \vdots & & \vdots & \vdots & \vdots  & \cdots \\
0 & 0 & 0 &  \dots  & 0 & 1 & 0 & \cdots \\
0 & 0 & 0 &  \dots  & 1 & 0 & 1 & \cdots \\
0 & 0 & 0 &  \dots & 0  & 1 & 0 & \cdots \\
\vdots & \vdots & \vdots & & \vdots & \vdots & \vdots  & \ddots \\
\end{pmatrix}
\]
has no eigenvalues in $\ell^p$, for $1\leq p < \infty$,  and hence the
system \eqref{eq:infinite} has no $p$-summable solution.

Let us compute the spectrum of $S$ in $\Lp{G}$. Let $f_j:=\chi_{v_j}$
and $g_j:=\chi_{w_j}$. Then it is clear that the set $\{ f_j  \, : \,
j \in \N \} \cup \{ g_j, \, : \,  j \in \N \}$ is a Schauder basis
for $\Lp{G}$. If we write the matrix of $S$ with respect to this
basis, it is clear that in can be written, in block form, as
\[
S=\left(
\begin{array}{c|c}
T & I \\
\hline
I & 0
\end{array}
\right),
\]
where $T$ is as defined above. To find the spectrum of $S$ we need to
find for what values of $\lambda$ the matrix
\[
S-\lambda=\left(
\begin{array}{c|c}
T -\lambda & I \\
\hline
I & -\lambda
\end{array}
\right),
\]
is invertible. 
If $\lambda=0$, then $S$ is invertible and
\[
S^{-1}=\left(
\begin{array}{c|c}
0 & I \\
\hline
I & -T
\end{array}
\right),
\]
Now assume $\lambda\neq 0$. It is well known and easily verified that a $2\times 2$ block
matrix of the form above is invertible if and only if
$T-\lambda+\frac{1}{\lambda}$ is invertible. In that case, in the style of
\cite[p. 18]{HoJo}, the inverse is
\[
(S-\lambda)^{-1}=\left(
\begin{array}{c|c}
(T -\lambda + \frac{1}{\lambda})^{-1} & \frac{1}{\lambda}(T -\lambda +
                                        \frac{1}{\lambda})^{-1}
  \\[.1cm]
\hline
\rule{0pt}{1\normalbaselineskip}
\frac{1}{\lambda }(T -\lambda + \frac{1}{\lambda})^{-1} &
                                                          -\frac{1}{\lambda}
                                                          +
                                                          \frac{1}{\lambda^2}
                                                          (T -\lambda
                                                          +
                                                          \frac{1}{\lambda})^{-1}
\end{array}
\right).
\]
Hence $\lambda \in \sigma(S)$ if and only $\lambda-\frac{1}{\lambda}
\in \sigma(T)$. Recall  $\sigma(T)=[-2,2]$ (see, for example,
\cite{Duren}). It can easily be checked that $\lambda-\frac{1}{\lambda}
\in [-2,2]$ if and only if $\lambda \in [-1-\sqrt{2},1-\sqrt{2}]\cup
[-1+\sqrt{2}, 1+\sqrt{2}]$. Hence
\[
\sigma(S)=[-1-\sqrt{2},1-\sqrt{2}]\cup [-1+\sqrt{2}, 1+\sqrt{2}].
\]

\end{document}